\documentclass[reqno]{amsart}
\usepackage{amssymb}
\usepackage{graphicx}

\usepackage[usenames, dvipsnames]{color}
\usepackage{verbatim}
\usepackage{mathrsfs}
\usepackage{bm}
\usepackage{cite}

\numberwithin{equation}{section}

\newtheorem{theorem}{Theorem}[section]
\newtheorem{corollary}[theorem]{Corollary}
\newtheorem{lemma}[theorem]{Lemma}
\newtheorem{prop}[theorem]{Proposition}

\theoremstyle{definition}
\newtheorem{remark}[theorem]{Remark}

\theoremstyle{definition}

\theoremstyle{definition}

\theoremstyle{definition}

\makeatletter
\def\dashint{\operatorname%
{\,\,\text{\bf-}\kern-.98em\DOTSI\intop\ilimits@\!\!}}
\makeatother

\def\\det{\text{\det}}

\def\Xint#1{\mathchoice
 {\XXint\displaystyle\textstyle{#1}}%
 {\XXint\textstyle\scriptstyle{#1}}%
 {\XXint\scriptstyle\scriptscriptstyle{#1}}%
 {\XXint\scriptscriptstyle\scriptscriptstyle{#1}}%
 \!\int}
\def\XXint#1#2#3{{\setbox0=\hbox{$#1{#2#3}{\int}$}
  \vcenter{\hbox{$#2#3$}}\kern-.5\wd0}}

\def\dashint{\Xint-}

\def\.5{\frac{1}{2}}

\newcommand{\RN}[1]{%
  \textup{\uppercase\expandafter{\romannumeral#1}}%
}

\renewcommand{\epsilon}{\varepsilon}

\newcounter{marnote}

\begin{document}

\title[A Nash-Kuiper theorem in a high codimension]{A Nash-Kuiper theorem for isometric immersions in a high codimension}


\author[Z.W. Zhao]{Zhiwen Zhao}

\address[Z.W. Zhao]{School of Mathematics and Physics, University of Science and Technology Beijing, Beijing 100083, China.}
\email{zwzhao365@163.com}


\date{\today} 


\maketitle
\begin{abstract}
This paper is devoted to investigating the isometric immersion problem of Riemannian manifolds in a high codimension. It has recently been demonstrated that any short immersion from an $n$-dimensional smooth compact manifold into $2n$-dimensional Euclidean space can be uniformly approximated by $C^{1,\theta}$ isometric immersions with any $\theta\in(0,1/(n+2))$ in dimensions $n\geq3$. In this paper, we improve the H\"{o}lder regularity of the constructed isometric immersions in the local setting, achieving $C^{1,\theta}$ for all $\theta\in(0,1/n)$ in odd dimensions and all $\theta\in(0,1/(n+1))$ in even dimensions. Moreover, we also establish explicit $C^{1}$ estimates for the isometric immersions, which indicate that the larger the initial metric error is, the greater the $C^{1}$ norms of the resulting isometric maps become, meaning that their slope become steeper.

\end{abstract}

\maketitle



\section{Introduction and main results}

The classical isometric immersion problem seeks an immersion $u$ from an $n$-dimensional Riemannian manifold $(\mathcal{M}^{n},g)$ into $\mathbb{R}^{d}$ such that the length of every $C^{1}$ curve is preserved under the mapping. In the language of Riemannian geometry, this means that the pullback metric $u^{\sharp}e$  coincides with the intrinsic metric $g$ on the manifold, where $e$ is the Euclidean metric on $\mathbb{R}^{d}$. In local coordinates, this relation is expressed as a system of $n_{\ast}:=n(n+1)/2$ partial differential equations in $d$ unknows as follows:
\begin{align}\label{S01}
\partial_{i}u\cdot\partial_{j}u=g_{ij}.	
\end{align}	
When $0<\partial_{i}u\cdot\partial_{j}u<g$ in the sense of quadratic forms, $u$ is called a short immersion. If $u$ is further assumed to be injective, then the isometric and short immersions become, respectively, an isometric embedding and a short embedding.

The problem of determining whether a Riemannian manifold can be isometrically immersed (or embedded) into some Euclidean space was first raised by Schl\"{a}fli \cite{S1871}, who conjectured that the system of \eqref{S01} is locally solvable provided that the target Euclidean space has dimension at least $n_{\ast}$, equal to the number of equations in \eqref{S01}. If both the manifold and its metric are real analytic, Janet \cite{J1927} proved that any $(\mathcal{M}^{2},g)$ can be locally isometrically embedded into $\mathbb{R}^{3}$. Later, E. Cartan \cite{C1928} generalized this result to higher dimensions, establishing local analytic isometric embeddings of $(\mathcal{M}^{n},g)\hookrightarrow\mathbb{R}^{n_{\ast}}$. When the target dimension $d<n_{\ast}$, since the system in \eqref{S01} becomes overdetermined, it had experienced a long time suspicion about whether the isometric embeddings still exist until Nash \cite{N1954} made a groundbreaking discovery by proving the existence of $C^{1}$ isometric embeddings in the case of $d\geq n+2$, revealing that the solvability of \eqref{S01} is governed solely by topological constraints. Subsequently, Kuiper \cite{K1955} addressed the low-codimension case of $d=n+1$ by constructing the perturbations of corrugation type, which are different from the spiral-type perturbations introduced by Nash \cite{N1954}. In fact, the Nash--Kuiper theorem established in \cite{N1954,K1955} is not only an existence result, but also exhibits a vast set of $C^{1}$ solutions. Such a profusion of solutions epitomizes a fundamental feature of Gromov’s $h$-principle \cite{G1973,G1986}, for which the isometric embedding problem is a special example. It is worth noting that the iteration scheme created by Nash \cite{N1954} can be regarded as an early prototype of the convex integration method. The latter was systematized and extended by Gromov \cite{G1973,G1986}, becoming a powerful tool for solving various nonlinear geometric problems and constructing wild solutions to a variety of equations in fluid dynamics.

Unlike the abundance of $C^{1}$ isometric embeddings in Nash--Kuiper theorem (often termed as the flexibility of isometric embeddings), classical rigidity results imply that smooth isometric embeddings into low-codimension Euclidean spaces are generally unique. For instance, the rigidity theorem of the Weyl problem in \cite{C1930,H1943} showed that a $C^{2}$ isometric immersion of a closed positively curved sphere into $\mathbb{R}^{3}$ is unique up to rigid motions. Such a sharp dichotomy between the flexibility of $C^{1}$ isometric immersions and the rigidity in the $C^{2}$ setting naturally leads to the question: does there exist a critical Hölder exponent $\theta_{c}\in(0,1)$ such that

$(i)$ for $\theta<\theta_{c}$, the Nash--Kuiper flexibility still persists;

$(ii)$ for $\theta>\theta_{c}$, isometric immersions in $C^{1,\theta}$ are rigid?

The investigation of $C^{1,\theta}$ isometric immersions originates from the seminal work of Borisov in the 1950s. He proved in \cite{B19581959} that the Cohn-Vossen--Herglotz rigidity remains valid for $C^{1,\theta}$ immersions provided $\theta>2/3$. In \cite{B1965} Borisov claimed that the Nash--Kuiper flexibility extends to $C^{1,\theta}$ whenever $\theta$ is smaller than $n_{\ast}$, now called the \emph{Borisov exponent}. A proof in the case of $n=2$ and $\theta<1/7$ was given in \cite{B2004} under the assumption of an analytic metric. Conti, De Lellis and Sz\'{e}kelyhidi \cite{CDS2012} confirmed Borisov's announcements for general dimension and general metric. In the presence of two-dimensional disks immersed into $\mathbb{R}^{3}$, De Lellis, Inauen and Sz\'{e}kelyhidi \cite{DIS2018} raised the H\"{o}lder exponent from $1/7$ to $1/5$ in virtue of the theory of conformal maps. The corresponding result in the global setting can be seen in \cite{CS2022}. Recently, Cao, Hirsch and Inauen \cite{CHI2025} further improved Borisov's exponent to any $\theta<1/(1+n^{2}-n)$ in the codimension-one case. In particular, the H\"{o}lder exponent of $n=2$ obtained in \cite{CHI2025} corresponds precisely to the Onsager critical exponent $1/3$ appearing in the context of non-uniqueness of weak solutions to the Euler equations. A similar dichotomy phenomena conjectured by Onsager \cite{O1949} is that weak solutions of the Euler equations with Hölder regularity $\theta<1/3$ may dissipate energy, whereas the solutions with $\theta > 1/3$ conserve energy. While the rigidity result was proved in \cite{CWT1994}, the resolution of the flexibility case eventually came in \cite{I2018}, based on a series of developments \cite{DS2017,BDIS2015,DS2013,DS2014} that built on the breakthrough ideas of \cite{DS2009}. The subsequent work \cite{BDSV2019} further constructed dissipative Euler solutions below the Onsager regularity threshold. Notably, the non-uniqueness constructions in these work rely on highly oscillatory building blocks such as Mikado flows and Beltrami waves, reflecting the same philosophy as that behind the classical Nash-spirals and Kuiper-corrugations. In fact, the high oscillation of these building blocks captures the sharp variations characteristic of turbulent structures observed in the physical world, which is the essential mechanism underlying the construction of non-unique solutions. In contrast, highly oscillatory Nash-spirals and Kuiper-corrugations in differential geometry are used to enlarge the induced metric of a short immersion, thereby reducing the metric error. This effect can be easily visualized through the short embedding of a one-dimensional curve into $\mathbb{R}^2$ or $\mathbb{R}^3$.

It is worth mentioning that the precise value of the critical H\"{o}lder exponent $\theta_{c}$ in the isometric immersion problem still remains open. We now briefly summarize how the H\"{o}lder regularity exponent arises in the local Nash–Kuiper convex integration method. The metric error is iteratively eliminated stage by stage. Each stage consists of $N$ steps for some fixed $N\in\mathbb{N}$, with each step involving the constructions of high-frequency low-amplitude perturbations in the form of Nash spirals or Kuiper corrugations. The number $N$ of steps in a stage leads to the resulting H\"{o}lder regularity exponent $1/(1+2N)$.  See Section 3 in \cite{CDS2012} for more detailed explanations. Therefore, for $n=2$, the exponent $1/3$ obtained in \cite{CHI2025} represents the best result achievable under the current technical framework. However, the recent work \cite{CI2024,DI2020} provided compelling evidence suggesting that $1/2$ may be the critical threshold in the two-dimensional case.

When the codimension is sufficiently large, the situation becomes different. In this case, one can construct isometric immersions with higher regularity. This topic was pioneered by Nash \cite{N1956}, who showed that any short immersion admits a uniform approximation by $C^{\infty}$ isometric immersions provided that the target dimension $d\geq3 n_{\ast}+4n$. The corresponding results in H\"{o}lder spaces was given by Jacobowitz \cite{J1972}. The target dimension in \cite{N1956} was extended to $d\geq n_{\ast}+2n+3$ in \cite{GR1970,G1986}, and further enhanced by M. G\"{u}nther \cite{G1989,G1990} to the bound $d\geq n_{\ast}+\max\{2n,n+5\}$. Yet, the smallest admissible target dimension remains an open problem. With regard to the case of less regular metrics, see \cite{K1978}.

Motivated by the classical Whitney's strong embedding theorem \cite{W1944}, in this paper we study the isometric immersion problem from a bounded connected $n$-dimensional domain into $\mathbb{R}^{2n}$, with the primary goal of establishing a local version of the Nash–Kuiper theorem in which the H\"{o}lder regularity exponent for the constructed isometric maps is improved compared to recent results in \cite{CS2025}. For that purpose, a lot of techniques and tools, such as Nash's iterative scheme \cite{N1954}, Kuiper-corrugations \cite{K1955}, the multiple perturbations \cite{K1978} and the iterative integration by parts \cite{CHI2025}, are employed.  

For any given $0<\varepsilon< 1/(n+1)$, denote
\begin{align}\label{OPM01}
	\theta:=\theta(n,\varepsilon)=
	\begin{cases}
		1/n-\varepsilon,&\text{if $n$ is odd,}\\
		1/(n+1)-\varepsilon,&\text{if $n$ is even}.  
	\end{cases}		
\end{align}
The main results in this paper are stated as follows.	
\begin{theorem}\label{Main01}
For $n\geq3$, let $\Omega\subset\mathbb{R}^{n}$ be a smooth bounded and simply connected domain with a $C^{1}$ Riemannian metric $g$. For any fixed $0<\varepsilon\ll1/(n+1)$, if $\bar{u}\in  C^{1,\varepsilon}(\overline{\Omega},\mathbb{R}^{2n})$ is a short immersion, then there exists an isometric immersion $u\in C^{1,\theta}(\overline{\Omega},\mathbb{R}^{2n})$ such that 
\begin{align*}
\|u-\bar{u}\|_{C^{0}(\overline{\Omega})}<\varepsilon,\;\, \|u\|_{C^{1}(\overline{\Omega})}\leq\|\bar{u}\|_{C^{1}(\overline{\Omega})}+\mathcal{G}_{\ast}+\varepsilon,\;\,\|\nabla(u-\bar{u})\|_{C^{0}(\overline{\Omega})}\leq\mathcal{K}_{\ast}+\varepsilon,	
\end{align*}
where $\theta,\mathcal{G}_{\ast},\mathcal{K}_{\ast}$ are, respectively, defined by \eqref{OPM01}, \eqref{AQ68} and \eqref{AQ69}.

\end{theorem}	

\begin{remark}
In comparison with the recent results in \cite{CS2025}, when $n\geq3$, we improve the Hölder exponent of the isometric maps from $1/(n+2)-\varepsilon$ to $1/n-\varepsilon$ in odd dimensions, and to $1/(n+1)-\varepsilon$ in even dimensions. As for $n=2$, the same exponent $1/3-\varepsilon$ as in \cite{CS2025} can be recovered by directly applying the current argument in Theorem \ref{Main01} with minor modification, and thus the proof is omitted here.	
	
\end{remark}	

\begin{remark}
In addition to establishing the uniform approximation estimates of isometric maps, we also provide stronger $C^{1}$ estimates on the isometric immersions, which clearly reveal the monotonicity of their $C^{1}$ norms with respect to the initial metric error, namely, the $C^{1}$ norm increases as the initial deviation from the target metric increases. This behavior is in accordance with the geometric intuition that larger initial deficits require sharper corrections in the approximation maps.
	
\end{remark}

\begin{remark}
As seen from the proof of Theorem \ref{Main01}, although the assumption of $0<\varepsilon\ll1/(n+1)$ can be relaxed to that of $0<\varepsilon<2/n^{2}$, the essential role of $\varepsilon$ lies in its arbitrary smallness, which ensures that the isometric approximation holds at arbitrarily small scales. 

\end{remark}

A slight modification of the proof of Theorem \ref{Main01} yields the same conclusion under more general assumptions as follows.
\begin{corollary}
For $n\geq3$, assume that $\Omega\subset\mathbb{R}^{n}$ is a smooth bounded and simply connected domain with a $C^{1}$ Riemannian metric $g$. For any given three small parameters $0<\varepsilon\ll1/(n+1),0<\kappa,\varsigma\ll1$, if $\bar{u}:(\overline{\Omega},g)\rightarrow\mathbb{R}^{2n}$ is a short immersion of class $ C^{1,\kappa}$, then there exists a $C^{1,\theta}$ isometric immersion $u:(\overline{\Omega},g)\rightarrow\mathbb{R}^{2n}$ depending only on $\varepsilon,\kappa,\varsigma,\|g\|_{C^{1}},\|\bar{u}\|_{C^{1,\kappa}}$ such that $\|u-\bar{u}\|_{C^{0}(\overline{\Omega})}<\varsigma,$ where $\theta$ is given by \eqref{OPM01}.
\end{corollary}	

A consequence of Theorem \ref{Main01} and Remark 8.1 in \cite{DIS2018} gives the following isometric version of Whitney's strong embedding theorem in the local setting.
\begin{corollary}
For $n\geq3$, let $\Omega\subset\mathbb{R}^{n}$ be a smooth bounded and simply connected domain endowed with a $C^{1}$ Riemannian metric $g$. Given any small constant $0<\varepsilon\ll1/(n+1)$, there exist infinitely many $C^{1,\theta}$ isometric embeddings of $(\overline{\Omega},g)$ into $\mathbb{R}^{2n}$, where $\theta$ is defined by \eqref{OPM01}.
	
\end{corollary}

The remainder of the paper is organized as follows. In Section \ref{SEC02}, we introduce some notations and gather several technical lemmas for later use. Section \ref{SEC03} is dedicated to the proof of Theorem \ref{Main01}, which is achieved by reducing its proof to the construction of a Cauchy sequence of smooth short immersions satisfying some iterative estimates with superexponential decay as stated in Proposition \ref{pro01}. In Section 4, we examine the limitations encountered when attempting to further raise the regularity exponent in even dimensions. Then a relevant matrix decomposition is presented in Proposition \ref{pro003} for potential future applications.

\section{Preliminaries}\label{SEC02}

\subsection{Notations and auxiliary estimates}
For $A\in\mathbb{R}^{n\times n}$, denote by $A^{t}$ the transpose of $A$. Write $\mathrm{sym}(A)=\frac{1}{2}(A+A^{t})$ and $\mathrm{Sym}_{n}=\{A\in\mathbb{R}^{n\times n}:A=A^{t}\}$.  Let $O(1)$ be some quantity satisfying that $|O(1)|\leq\,C$. Here and in the following, unless otherwise stated, the constant $C$ represents a universal positive constant depending only on $n,\|g\|_{1},\|\bar{u}\|_{1+\varepsilon},\|g-\bar{u}^{\sharp}e\|_{0}$ and the number of derivatives, whose value may vary from line to line.

For a real-valued, vector-valued, or general tensor-valued function $f$, denote $\|f\|_{0}:=\sup\limits_{x\in\Omega}|f|$ and define the H\"{o}lder seminorms as follows: 
\begin{align*}
	[f]_{k}=\sum_{|\beta|=k}\|\nabla^{\beta}f\|_{0},\quad[f]_{k+\theta}=\sum_{|\beta|=k}\sup_{x\neq y,x,y\in\Omega}\frac{|\nabla^{\beta}f(x)-\nabla^{\beta}f(y)|}{|x-y|^{\theta}},	
\end{align*}	
where $k\in\mathbb{N}$, $\theta\in(0,1)$, $\beta$ is a multi-index. Then the H\"{o}lder norms are listed as follows:
\begin{align*}
	\|f\|_{k}=\sum^{k}_{j=0}[f]_{j},\quad\|f\|_{k+\theta}=\|f\|_{k}+[f]_{k+\theta}.	
\end{align*}	 
We now state the following required mollification estimates.
\begin{lemma}\label{LEM08}
	For $k\in\mathbb{N}$ and $0<\theta\leq1$, we have
	\begin{align}\label{A26}
		\|f\ast\varphi_{\ell}\|_{k}\leq\|f\|_{k},\quad[f-f\ast\varphi_{\ell}]_{k}\leq\ell^{\theta}[f]_{k+\theta},
	\end{align}
	and
	\begin{align}\label{A27}
		\begin{cases}
			\|f-f\ast\varphi_{\ell}\|_{r}\leq C\ell^{2-r}[f]_{2},&0\leq r\leq2,\\
			\|(fg)\ast\varphi_{\ell}-(f\ast\varphi_{\ell})(g\ast\varphi_{\ell})\|_{r}\leq C\ell^{2\theta-r}\|f\|_{\theta}\|g\|_{\theta},&r\geq0,
		\end{cases}	
	\end{align}	
	where $f,g$ are two	real-valued, vector-valued, or general tensor-valued functions,  $\varphi_{\ell}$ is the standard mollifying kernel on $\mathbb{R}^{n}$, and the constant $C$ depends only on $r,\theta,\varphi$.
\end{lemma}	
\begin{proof}
	\eqref{A26} is an immediate consequence of Bochner's integral inequality and Young's convolution inequality. The proof of \eqref{A27} can be seen in Lemma 1 in \cite{CDS2012} and Lemma 3.3 in \cite{DIS2018}.
	
\end{proof}

\subsection{Decomposition of symmetric matrices}

Denote by $\{e_{i}\}_{i=1}^{n}$ the standard basis of $\mathbb{R}^{n}$. Write
\begin{align}\label{XI}
\{\xi_{i}:i=1,2,...,n_{\ast}\}=\bigg\{\frac{e_{i}+e_{j}}{|e_{i}+e_{j}|}:1\leq i\leq j\leq n\bigg\},\quad	\xi_{i}=e_{i},\;\,i=1,2,...,n.
\end{align}	
Then $\{\xi_{i}\otimes\xi_{i}\}_{i=1}^{n_{\ast}}$ forms a basis of $\mathrm{Sym}_{n}$. Define
\begin{align}\label{M01}
h_{\ast}=\sum^{n_{\ast}}_{i=1}\xi_{i}\otimes\xi_{i}.	
\end{align}	
A direct computation shows that $h_{\ast}$ is a positive definite matrix. Recalling Lemma 3 in \cite{CDS2012} and Lemma 2.1 in \cite{CHI2025}, we have
\begin{lemma}\label{lem001}
There exist a small constant $\sigma_{\ast}=\sigma_{\ast}(n)>0$ and linear maps $L_{i}:\mathrm{Sym}_{n}\rightarrow\mathbb{R}$ such that
\begin{align*}
h=\sum^{n_{\ast}}_{i=1}L_{i}(h)\xi_{i}\otimes\xi_{i},\quad\text{for any $h\in\mathrm{Sym_{n}}$,}	
\end{align*}	
and, furthermore, $L_{i}(h)\geq\sigma_{\ast}$ if $|h-h_{\ast}|\leq2\sigma_{\ast}$.

\end{lemma}	

Applying the proof of Lemma 2.2 in \cite{CHI2025} with a slight modification, we obtain the following K\"{a}ll\'{e}n-type decomposition.
\begin{lemma}\label{lem002}
Assume that $n\geq2$, $1\leq\mu_{0}\leq\mu_{1}\leq\cdots\leq\mu_{n}$. Then there exists a small constant $0<\sigma_{0}<\sqrt{\sigma_{\ast}}$ such that if $h\in C^{\infty}(\overline{\Omega},\mathrm{Sym}_{n})$ satisfies
\begin{align*}
\|h-h_{\ast}\|_{0}+\frac{\mu_{0}}{\mu_{1}}\leq2\sigma_{0},\quad\|h\|_{i}\leq\mu_{0}^{i},\quad i\geq1,
\end{align*}		
then for any $ j\geq 0$, there exist $a^{j}=(a_{1}^{j},...,a_{n_{\ast}}^{j})\in C^{\infty}(\overline{\Omega},\mathbb{R}^{n_{\ast}})$ and $\mathcal{E}^{j}\in C^{\infty}(\overline{\Omega},\mathrm{Sym}_{n})$ such that
\begin{align}\label{QE01}
h=\sum^{n_{\ast}}_{i=1}(a_{i}^{j})^{2}\xi_{i}\otimes\xi_{i}+\sum^{n}_{i=1}\mu_{i}^{-2}\nabla a_{i}^{j}\otimes\nabla a_{i}^{j}+\mathcal{E}^{j},	
\end{align}	
and 
\begin{align*}
\begin{cases}
a_{i}^{j}\geq\sqrt{\sigma_{\ast}},&i=1,...,n_{\ast},\\
\|a^{j}\|_{k}\leq C(n)\mu_{0}^{k},&k\geq0,\\
\|\mathcal{E}^{j}\|_{k}\leq C(n)(\mu_{0}\mu_{1}^{-1})^{2(j+1)}\mu_{0}^{k},&k\geq0,
\end{cases}		
\end{align*}	
where $\{\xi_{i}\}_{i=1}^{n_{\ast}}$ and $\sigma_{\ast}$ are, respectively, given by \eqref{XI} and Lemma \ref{lem001}.
\end{lemma}	

\begin{remark}
From Lemma \ref{lem002}, we see that the first term in \eqref{QE01} is the leading part of $h$, while the remaining two terms can be regarded as small lower-order perturbations.  	
	
\end{remark}	

The following lemma is adapted from Lemma 2.3 of \cite{CHI2025} with slight adjustments, which plays a crucial role in the proof of Lemma \ref{lem08} below.
\begin{lemma}\label{lem06}
Set $n\geq2$. For any given $c_{\ast}\in\mathbb{R}\setminus\{0\}$ and every $i=1,2,...,n$, the linear map 
\begin{align*}
\Phi_{i}:\mathbb{R}^{n}\times\mathbb{R}^{n_{\ast}-n}\longrightarrow&\mathrm{Sym}_{n},\\	
(\alpha,\beta)\longmapsto &c_{\ast}\alpha\odot\xi_{i}+\sum^{n_{\ast}}_{j=n+1}\beta_{j-n}\xi_{j}\otimes\xi_{j}
\end{align*}		
is an isomorphism, where $\alpha\odot\xi_{i}:=\frac{1}{2}(\alpha\otimes \xi_{i}+\xi_{i}\otimes\alpha),$ and $\{\xi_{i}\}_{i=1}^{n_{\ast}}$ is defined by \eqref{XI}.
	
\end{lemma}	
\begin{remark}
The primary objective of the decomposition in Lemma \ref{lem06} is to isolate and explicitly represent the contribution of a symmetric matrix along any chosen direction, while parameterizing the remaining degrees of freedom in the complementary subspace.

\end{remark}
Building on the argument in Lemma 2.3 of \cite{CHI2025}, we provide a revised version of their proof.
\begin{proof}
By symmetry and the rank-nullity theorem, it suffices to show that $\Phi_{1}$ is injective. Assume that $(\alpha,\beta)=(\alpha_{1},...,\alpha_{n},\beta_{1},...,\beta_{n_{\ast}-n})\in\mathbb{R}^{n}\times\mathbb{R}^{n_{\ast}-n}$ satisfies 
\begin{align}\label{QE90}
	\Phi_{1}(\alpha,\beta)=c_{\ast}\alpha\odot\xi_{1}+\sum^{n_{\ast}}_{j=n+1}\beta_{j-n}\xi_{j}\otimes\xi_{j}=0.
\end{align}		
For $j=2,...,n$, by projecting both sides onto $\xi_{j}$, we have from \eqref{XI} that
\begin{align}\label{Z96}
0=&\frac{c_{\ast}}{2}\alpha_{j}\xi_{1}+\sum_{\xi_{l}\cdot\xi_{j}\neq0}\beta_{l-n}\xi_{l}=\frac{c_{\ast}}{2}\alpha_{j}\xi_{1}+\sum_{k=1,k\neq j}^{n}\frac{\gamma_{k}^{(j)}}{\sqrt{2}}(\xi_{k}+\xi_{j})\notag\\
=&\bigg(\frac{c_{\ast}}{2}\alpha_{j}+\frac{\gamma_{1}^{(j)}}{\sqrt{2}}\bigg)\xi_{1}+\sum_{k=2,k\neq j}^{n}\frac{\gamma_{k}^{(j)}}{\sqrt{2}}\xi_{k}+\bigg(\sum_{k=1,k\neq j}^{n}\frac{\gamma_{k}^{(j)}}{\sqrt{2}}\bigg)\xi_{j},
\end{align}	
where 
\begin{align*}
\big\{\gamma_{k}^{(j)}\big\}_{k=1,k\neq j}^{n}=\{\beta_{l-n}:\xi_{l}\cdot\xi_{j}\neq0,n+1\leq l\leq n_{\ast}\}.
\end{align*}
Remark that the set $\big\{\gamma_{k}^{(j)}\big\}_{k=1,k\neq j}^{n}$ is actually obtained by permuting the elements of $\{\beta_{l-n}:\xi_{l}\cdot\xi_{j}\neq0,n+1\leq l\leq n_{\ast}\}$ so that \eqref{Z96} holds. Using the linear independence of $\{\xi_{i}\}_{i=1}^{n}$, we obtain that for $j=2,...,n,$
\begin{align*}
\frac{c_{\ast}}{2}\alpha_{j}+\frac{\gamma_{1}^{(j)}}{\sqrt{2}}=\sum_{l=1,l\neq j}^{n}\gamma_{l}^{(j)}=0,\quad\gamma_{k}^{(j)}=0,\;k\geq2,\,k\neq j,	
\end{align*}	
which leads to that
\begin{align}\label{E1}
	\alpha_{j}=\gamma_{k}^{(j)}=0,\quad k\geq1,\,k\neq j,\,j=2,...,n.
\end{align}	
Note that for any $2\leq j_{1}<j_{2}\leq n$, we have $\gamma_{j_{2}}^{(j_{1})}=\gamma_{j_{1}}^{(j_{2})}$, which implies that the sets $\big\{\gamma_{k}^{(j_{1})}\big\}_{k=1,k\neq j_{1}}^{n}$ and $\big\{\gamma_{k}^{(j_{2})}\big\}_{k=1,k\neq j_{2}}^{n}$ share only one common element. This, together with \eqref{E1}, gives that
\begin{align}\label{Z98}
\beta=0,\quad\alpha_{j}=0,\;\,j=2,...,n.	
\end{align}	 

It remains to prove that $\alpha_{1}=0$. Similarly, applying the dot product with $\xi_{1}$ on both sides of \eqref{QE90}, it follows from \eqref{QE90} that
\begin{align*}
0=&\frac{c_{\ast}}{2}(\alpha+\alpha_{1}\xi_{1})+\sum_{\xi_{l}\cdot\xi_{1}\neq0}\beta_{l-n}\xi_{l}=\frac{c_{\ast}}{2}(\alpha+\alpha_{1}\xi_{1})+\sum^{n}_{k=2}\frac{\gamma_{k}^{(1)}}{\sqrt{2}}(e_{k}+e_{1})\notag\\
=&\bigg(c_{\ast}\alpha_{1}+\sum^{n}_{k=2}\frac{\gamma_{k}^{(1)}}{\sqrt{2}}\bigg)\xi_{1}+\sum^{n}_{k=2}\bigg(\frac{c_{\ast}}{2}\alpha_{k}+\frac{\gamma_{k}^{(1)}}{\sqrt{2}}\bigg)\xi_{k},
\end{align*}		
where
\begin{align*}
	\big\{\gamma_{k}^{(1)}\big\}_{k=2}^{n}=\{\beta_{l-n}:\xi_{l}\cdot\xi_{1}\neq0,n+1\leq l\leq n_{\ast}\}.
\end{align*}	
It then folllows from the linear independence of $\{\xi_{i}\}_{i=1}^{n}$ again that
\begin{align}\label{Z100}
c_{\ast}\alpha_{1}+\sum^{n}_{l=2}\frac{\gamma_{l}^{(1)}}{\sqrt{2}}=0,\quad	\frac{c_{\ast}}{2}\alpha_{k}+\frac{\gamma_{k}^{(1)}}{\sqrt{2}}=0,\;\,k=2,...,n.
\end{align}		
Then substituting \eqref{Z98} into \eqref{Z100}, we have $\alpha=0$ and $\beta=0.$ 	The proof is complete.

\end{proof}	

Based on Lemma \ref{lem06} and applying the proof of Proposition 2.4 in \cite{CHI2025} with minor modification, we have
\begin{lemma}\label{lem08}
Assume that $\lambda\geq\mu\geq1$, $K\geq0$, $\gamma\in C^{\infty}(\mathbb{S}^{1})$ is a smooth periodic function with zero mean, and $\mathcal{Q}\in C^{\infty}(\overline{\Omega},\mathrm{Sym}_{n})$ satisfies 
\begin{align*}
\|\mathcal{Q}\|_{k}\leq K\mu^{k},\quad k\geq 0.	
\end{align*}	
Then for any fixed $j\geq 0$, there exist $w^{j}\in C^{\infty}(\overline{\Omega},\mathbb{R}^{n})$, $\mathcal{G}^{j}\in C^{\infty}(\overline{\Omega},\mathrm{Sym}_{n})$, $\mathcal{R}^{j}\in C^{\infty}(\overline{\Omega},\mathrm{Sym}_{n})$, and $\gamma^{j}\in C^{\infty}(\mathbb{S}^{1})$ with $\dashint\gamma^{j}=0$ such that for $i=1,...,n$,
\begin{align}\label{W980}
\gamma(\lambda x\cdot\xi_{i})\mathcal{Q}=2\mathrm{sym}(\nabla w^{j})+\gamma^{j}(\lambda x\cdot\xi_{i})(\mu\lambda^{-1})^{j}\mathcal{G}^{j}+\mathcal{R}^{j},	
\end{align}		
where $\{\xi_{i}\}_{i=1}^{n_{\ast}}$ is given by \eqref{XI}, $\mathcal{R}^{j}\in\mathrm{span}\{\xi_{i}\otimes\xi_{i}:i=n+1,...,n_{\ast}\}$, and
\begin{align*}
\|w^{j}\|_{k}\leq CK\lambda^{k-1},\quad\|\mathcal{G}^{j}\|_{0}\leq CK\mu^{k},\quad\|\mathcal{R}^{j}\|_{k}\leq CK\lambda^{k},\quad k\geq0,	
\end{align*}		
for some $C=C(n,\gamma)>0$. 
\end{lemma}	
\begin{remark}
When $\lambda>\mu$ and $j\gg1$, the decomposition in \eqref{W980} shows that up to a small error, a high-frequency oscillatory symmetric matrix field can be split into the linear symmetric gradient of some perturbation $w^{j}$ plus a nonlinear structural complement $\mathcal{R}^{j}$ with respect to $w^{j}$. When $0<K\ll1$, these two terms become controllable perturbations. The linear symmetric gradient term can be directly incorporated into the constructed high-frequency low-amplitude perturbations, while the nonlinear structural complement will be handled as a perturbation in the complementary subspace. This outlines how the iterative integration by parts technique is performed. 

\end{remark}

We end this section by recalling the existence of normal vectors for the given embedding in \cite{CI2024}.
\begin{lemma}\label{lem012}
Assume that $\Omega\subset\mathbb{R}^{n}$ is a simply connected domain. For $n,d,N\in\mathbb{N}$ and $n<d$, let $u\in C^{N+1}(\overline{\Omega},\mathbb{R}^{d})$	be an embedding such that $\gamma^{-1}\mathrm{Id}\leq\nabla u^{t}\nabla u\leq\gamma\mathrm{Id}$ for some $\gamma>1$. Then there exists a family of normal vectors $\{\eta_{i}\}_{i=1}^{d-n}\subset C^{N}(\overline{\Omega},\mathbb{R}^{d})$ such that for $i,j=1,...,d-n$ and $0\leq k\leq N$,
\begin{align*}
\eta_{i}^{t}\eta_{j}=\delta_{ij},\quad \nabla u^{t}\eta_{i}=0,\quad[\eta_{i}]_{k}\leq C(k,\gamma)(1+[u]_{k+1}),	
\end{align*}	
where $\delta_{ij}$ is the Kronecker symbol: $\delta_{ij}=0$ for $i\neq j$, $\delta_{ij}=1$ for $i=j$.
\end{lemma}

\section{The proof of Theorem \ref{Main01}}\label{SEC03}
For $m\geq0$, denote
\begin{align}\label{W009}
	\Omega_{m}=\{x\in\mathbb{R}^{n}:\mathrm{dist}(x,\Omega)<2^{-m-1}\}.	
\end{align}		
By the classical extension arguments (see e.g. \cite{W1934}), we assume without loss of generality that  $\bar{u}$ is a short $ C^{1,\varepsilon}$ immersion from a slightly larger region $\overline{\Omega}_{0}$ into $\mathbb{R}^{2n}$ and $g$ is a $C^{1}$ Riemannian metric on $\overline{\Omega}_{0}$, which compensates for reduction of the region caused by the following mollification step.

For $n\geq3$ and any fixed $0<\varepsilon\ll1/(n+1)$, introduce the following constants:
\begin{align}\label{W90}
	b=1+\frac{\varepsilon (n-1)}{2},\quad \vartheta=b(\theta+\varepsilon)(1-(n-1)\varepsilon),
\end{align}	
where $\theta$ is defined by \eqref{OPM01}. For $m\geq0$, $a\gg1$, define
\begin{align}\label{M06}
\delta_{m}=\delta_{\ast}a^{-2\vartheta b^{m}},\quad\lambda_{m}=\lambda_{\ast}a^{b^{m+1}+\tau},\quad\tau=2\vartheta b(N_{\ast}+1)+\vartheta-b,
\end{align}
where $\delta_{\ast},N_{\ast},\lambda_{\ast}$ are, respectively, defined by \eqref{W08}--\eqref{W16} and \eqref{W09} below. Denote
\begin{align*}
g_{m}=g-\delta_{m+1}h_{\ast},\quad m\geq0,
\end{align*}
where the matrix $h_{\ast}$ is defined by \eqref{M01}. Remark that $g_{m}$ is defined so as to preserve the shortness of the following approximating sequence at each step. In order to obtain an isometric approximation within any prescribed error tolerance in Theorem \ref{Main01}, the idea is to construct a Cauchy sequence of short immersions with the step-by-step reduction of the metric error. To be specific, we need to establish the iterative proposition as follows.
\begin{prop}\label{pro01}
For $m\geq0$, let $u_{m}\in C^{\infty}(\overline{\Omega}_{m},\mathbb{R}^{2n})$ be a short immersion satisfying that
\begin{align*}
\begin{cases}		
\|g_{m}-u^{\sharp}_{m}e\|_{0}\leq\sigma_{0}\delta_{m+1},\quad\|\nabla^{2} u_{m}\|_{0}\leq\delta_{m}^{1/2}\lambda_{m},\\
\|u_{m}\|_{1}\leq\|\bar{u}\|_{1}+\mathcal{G}_{\ast}+\big(\sum^{m}_{i=0}2^{-i}-1\big)\varepsilon.
\end{cases}	
\end{align*}	
Then there exists a short immersion $u_{m+1}\in C^{\infty}(\overline{\Omega}_{m+1},\mathbb{R}^{2n})$ such that
\begin{align*}
\|u_{m+1}\|_{1}\leq\|\bar{u}\|_{1}+\mathcal{G}_{\ast}+\bigg(\sum^{m+1}_{i=0}\frac{1}{2^{i}}-1\bigg)\varepsilon,	
\end{align*}
and	
\begin{align*}
\begin{cases}
\|g_{m+1}-u^{\sharp}_{m+1}e\|_{0}\leq\sigma_{0}\delta_{m+2},\quad\|\nabla^{2} u_{m+1}\|_{0}\leq\delta_{m+1}^{1/2}\lambda_{m+1},	\\
\|u_{m+1}-u_{m}\|_{0}\leq\delta_{m+1}^{1/2},\quad\|\nabla(u_{m+1}-u_{m})\|_{0}\leq C_{\ast}\delta_{m+1}^{1/2},
\end{cases}
\end{align*}	
where $\sigma_{0},C_{\ast},\mathcal{G}_{\ast}$ are given by Lemma \ref{lem002}, \eqref{AQ56} and \eqref{AQ68}, respectively.
\end{prop}	

\begin{remark}
From the proofs of Proposition \ref{pro01} and Theorem \ref{Main01}, we see that the value of the parameter $a$ depends only on $\varepsilon,n,\|g\|_{1},\|\bar{u}\|_{1+\varepsilon},\|g-\bar{u}^{\sharp}e\|_{0}$, but is independent of the sequence index $m$. In particular, if $\varepsilon\rightarrow0,$ then $a\rightarrow\infty$. More importantly, apart from the parameter $a$, the values of all other introduced parameters including $b,\vartheta,\alpha,J$ are given explicitly in \eqref{W90} and \eqref{W901}. This improvement not only clarifies their precise dependence on the accuracy parameter $\varepsilon$ and the dimension $n$, but also greatly facilitates the reader’s understanding for the applied techniques and methods.

\end{remark}	

\begin{remark}\label{REM01}
Using the interpolation inequality, we have
\begin{align*}
\|u_{m+1}-u_{m}\|_{1+\theta}\leq C&\|u_{m+1}-u_{m}\|_{1}^{1-\theta}\|u_{m+1}-u_{m}\|_{2}^{\theta}\leq Ca^{\tau\theta-b^{m+2}(\frac{\vartheta}{b}-\theta)}.	
\end{align*}	
From \eqref{W90}, we have $\theta<\frac{\vartheta}{b}<\theta+\varepsilon$, which shows that $\tau\theta-b^{m+2}(\frac{\vartheta}{b}-\theta)<0$ for any $m>\log_{b}\frac{\tau\theta}{b(\vartheta-b\theta)}$.
\end{remark}
\subsection{Mollification}
 Denote
\begin{align*}
h_{m}:=\frac{g\ast\varphi_{\ell}-(u_{m}\ast\varphi_{\ell})^{\sharp}e-\delta_{m+2}h_{\ast}}{\delta_{m+1}},
\end{align*}	
where $\varphi_{\ell}$ denotes the standard mollifying kernel on $\mathbb{R}^{n}$. Starting from the mollified $u_{m}\ast\varphi_{\ell}$, we aim to construct a new map $u_{m+1}$ satisfying the iterative requirements in Proposition \ref{pro01}. Note that 
\begin{align*}
g_{m+1}	-u_{m+1}^{\sharp}e=g-g\ast\varphi_{\ell}+\delta_{m+1}h_{m}+(u_{m}\ast\varphi_{\ell})^{\sharp}e-u_{m+1}^{\sharp}e.
\end{align*}	
For later use, we define two constants as follows:
\begin{align}\label{W901}
	\alpha=\frac{\varepsilon^{2}}{10},\quad J=\Big[\frac{4}{(n-1)\varepsilon}\Big]+3.	
\end{align}	
Here and below, the notation $[\cdot]$ denotes the integer part of a real number. For the purpose of not increasing the metric error, we pick the value of $\ell$ as follows:
\begin{align}\label{W018}
0<\ell=\frac{\delta_{m+1}^{1/2}}{\delta_{m}^{1/2}\lambda_{m}a^{\alpha}}<2^{-m-1},
\end{align}
where $\alpha$ is defined by \eqref{W901}. Using Lemma \ref{LEM08}, we have
\begin{align*}
&\delta_{m+2}^{-1}\|g-g\ast\varphi_{\ell}\|_{0}\leq \delta_{m+2}^{-1}\|g\|_{1}\ell\leq (\lambda_{\ast}\delta_{\ast})^{-1}\|g\|_{1}a^{-(\tau+\alpha)-b^{m+1}(1-\vartheta(2b+1/b-1))}\notag\\
&\leq(\lambda_{\ast}\delta_{\ast})^{-1}\|g\|_{1}a^{-(\tau+\alpha)-b(1-\vartheta(2b+1/b-1))}\leq \frac{\sigma_{0}}{2},
\end{align*}
and
\begin{align*}
&\|h_{m}-h_{\ast}\|_{0}\notag\\
&\leq\delta^{-1}_{m+1}\big(\delta_{m+2}\|h_{\ast}\|_{0}+\|(g_{m}-u_{m}^{\sharp}e)\ast\varphi_{\ell}\|_{0}+\|(u_{m}^{\sharp}e)\ast\varphi_{\ell}-(u_{m}\ast\varphi_{\ell})^{\sharp}e\|_{0}\big)\notag\\
&\leq \lambda_{m}^{-2\vartheta(b-1)}\|h_{\ast}\|_{0}+\sigma_{0}+C\delta^{-1}_{m+1}\ell^{2}\|\nabla u_{m}\|_{1}^{2}\notag\\
&\leq a^{-2\vartheta b(b-1)}\|h_{\ast}\|_{0}+\sigma_{0}+Ca^{-2\alpha}\leq2\sigma_{0}.
\end{align*}
Analogously, we obtain that for $k\geq1$,
\begin{align*}
\|h_{m}\|_{k}\leq&\delta_{m+1}^{-1}\big(\|(g-u_{m}^{\sharp}e)\ast\varphi_{\ell}\|_{k}+\|(u_{m}^{\sharp}e)\ast\varphi_{\ell}-(u_{m}\ast\varphi_{\ell})^{\sharp}e\|_{k}\big)\notag\\	
\leq& C\ell^{-k}+C\frac{\ell^{2-k}\delta_{m}\lambda_{m}^{2}}{\delta_{m+1}}\leq C\ell^{-k}.
\end{align*}	
Therefore, based on these above facts, the problem reduces to finding a new immersion $u_{m+1}$ such that the difference of $(u_{m}\ast\varphi_{\ell})^{\sharp}e-u_{m+1}^{\sharp}e$ can be used to cancel the previous gap of  $\delta_{m+1}h_{m}$ up to a metric error of order $o(1)\delta_{m+2}$, where $o(1)$ denotes a quantity vanishing as the parameter $a\rightarrow\infty$.

Denote 
\begin{align}\label{W86}
 \lambda_{m,0}=\ell^{-1},\quad\Lambda=\frac{\delta_{m+1}a^{\alpha}}{\delta_{m+2}},
\end{align}
and
\begin{align*}
\lambda_{m,i}=
\begin{cases}
\sqrt[J]{\Lambda}\lambda_{m,i-1},&\mathrm{for}\; i=1,...,n,\\
 \Lambda\lambda_{m,i-1},&\mathrm{for}\;i=n+1,...,[\frac{3n}{2}],
 \end{cases}\quad\Big[\frac{3n}{2}\Big]=
\begin{cases}
\frac{3n-1}{2},&\text{if $n$ is odd},\\
\frac{3n}{2},&\text{if $n$ is even},	
\end{cases}	 
\end{align*}	
where the values of $\alpha,J$ are given by \eqref{W901}. To make the construction of $u_{m+1}$ more transparent, we first apply Lemma \ref{lem002} with $\mu_{i}=\lambda_{m,i},\,i=0,1,...,n_{\ast}$ to decompose the deficit into a finite sum of primitive metrics consisting of rank-one tensors with smooth positive coefficients as follows: 
\begin{align}\label{AQ06}
	h_{m}=\sum^{n_{\ast}}_{i=1}a_{i}^{2}\xi_{i}\otimes\xi_{i}+\sum^{n}_{i=1}\lambda_{m,i}^{-2}\nabla a_{i}\otimes\nabla a_{i}+\mathcal{E},	
\end{align}
where $a_{i}:=a_{i}^{J}$ and $\mathcal{E}:=\mathcal{E}^{J}$ satisfy that for $1\leq i\leq n_{\ast}$ and $k\geq0$,
\begin{align*}
a_{i}\geq\sqrt{\sigma_{\ast}}/2,\quad\|a_{i}\|_{k}\leq C(n)\lambda_{m,0}^{k},\quad\|\mathcal{E}\|_{k}\leq C(n)\Lambda^{-2(J+1)/J}\lambda_{m,0}^{k}.
\end{align*}	
This decomposition further reduces the task to constructing a sequence of perturbations designed to eliminate the individual primitive metrics.

\subsection{Constructions of the first $n$ perturbations}\label{sub01}
Let $u_{m,0}:=u_{m}\ast\varphi_{\ell}$. Then we obtain from \eqref{A27} that
\begin{align}\label{Q01}
\|u_{m,0}-u_{m}\|_{1}\leq C\ell\|u_{m}\|_{2}\leq C_{1}(n)\delta_{m+1}^{1/2}a^{-\alpha}.
\end{align}
For simplicity, denote 
\begin{align*}
\gamma_{i,1}=\sin(\lambda_{m,i}x\cdot\xi_{i}),\quad \gamma_{i,2}=\cos(\lambda_{m,i}x\cdot\xi_{i}),	\quad 1\leq i\leq n.
\end{align*}
Inspired by a new corrugation ansatz of Kuiper-type given in \cite{CHI2025}, we introduce the adapted Nash-type spirals in the high-codimension setting that match the iterative decomposition in Lemma \ref{lem08} as follows: for $1\leq i\leq n$, 
\begin{align*}
u_{m,i}=u_{m,i-1}+\frac{\delta_{m+1}^{1/2}a_{i}}{\lambda_{m,i}}	\big(\gamma_{i,1}\zeta_{i-1}+\gamma_{i,2}\eta_{i-1}\big)+\delta_{m+1}\mathcal{F}_{i-1}w_{i-1},
\end{align*}	
where $w_{i-1}$ is given by \eqref{AQD001} below, 
\begin{align}\label{AQ08}
\begin{cases}
\nabla u_{m,i-1}^{t}\zeta_{i-1}=\nabla u_{m,i-1}^{t}\eta_{i-1}=0,\; \zeta_{i-1}\cdot\eta_{i-1}=0,\;|\zeta_{i-1}|=|\eta_{i-1}|=1,\\
\mathcal{F}_{i-1}=\nabla u_{m,i-1}(u_{m,i-1}^{\sharp}e)^{-1},\quad\|w_{i-1}\|_{k}\leq C\Lambda^{-1/J}\lambda_{m,i}^{k-1},\quad k\geq0.
\end{cases}	
\end{align}	 
Here the normal vectors $\zeta_{i-1},\eta_{i-1}$ are given by Lemma \ref{lem012}. Based on the fact that
\begin{align*}
\begin{cases}	
\|u_{m,0}\|_{k}\leq C\lambda_{m,0}^{k-2}\|u_{m,0}\|_{2}\leq C\lambda_{m,0}^{k-2}\delta_{m}^{1/2}\lambda_{m},&k\geq2,\\
\delta_{m}^{1/2}\lambda_{m}\leq\delta_{m+1}^{1/2}\lambda_{m,0},\quad \|a_{i}\|_{k}\leq C\lambda_{m,0}^{k},&	1\leq i\leq n_{\ast},\;k\geq 0,
 \end{cases}
\end{align*}
it follows from a straightforward calculation that 
\begin{align}\label{Q02}
\|u_{m,n}-u_{m,0}\|_{k}\leq C_{2}(n)\delta_ {m+1}^{1/2}\lambda_{m,1}^{k-1},\quad k=0,1,
\end{align}
and, for $1\leq i\leq n,\,k\geq0,$
\begin{align}\label{AQ09}
\|u_{m,i}\|_{k}\leq C(1+\delta_{m+1}^{1/2}\lambda_{m,i}^{k-1}),
\end{align}	
and
\begin{align}\label{AQ10}
\|\zeta_{i-1}\|_{k},\|\eta_{i-1}\|_{k},\|\mathcal{F}_{i-1}\|_{k}\leq C\|u_{m,i-1}\|_{k+1}\leq C(1+\delta_{m+1}^{1/2}\lambda_{m,i-1}^{k}).	
\end{align}	
Then we have
\begin{lemma}\label{lem006}
There exists a function $\mathcal{R}\in C^{\infty}(\overline{\Omega}_{m},\mathrm{Sym}_{n})$ such that	
\begin{align*}
u_{m,n}^{\sharp}e=&u_{m,0}^{\sharp}e+\delta_{m+1}\sum^{n}_{i=1}\big(a_{i}^{2}\xi_{i}\otimes\xi_{i}+\lambda_{m,i}^{-2}\nabla a_{i}\otimes\nabla a_{i}\big)\notag\\
&+\delta_{m+1}\mathcal{R}+O(1)\delta_{m+1}\Lambda^{-1/J}\big(\Lambda^{-1}+\delta_{m+1}^{1/2}\big),	
\end{align*}	
where $\mathcal{R}\in\mathrm{span}\{\xi_{i}\otimes\xi_{i}:i=n+1,...,n_{\ast}\}$ satisfies
\begin{align*}
\|\mathcal{R}\|_{k}\leq	C\Lambda^{-1/J}\lambda_{m,n}^{k},\quad k\geq0.
\end{align*}	
\end{lemma}	
\begin{remark}
From the proof of Lemma \ref{lem006},  we see that the leading error terms consist of a low-frequency component $\delta_{m+1}\lambda_{m,i}^{-2}\nabla a_{i}\otimes\nabla a_{i}$ and a high-frequency component $\delta_{m+1}\sum^{2}_{j=1}\gamma_{i,j}\mathcal{Q}_{i,j}$ in the following. The low-frequency part is absorbed as a perturbation into the geometric decomposition in \eqref{AQ06}, while the high-frequency portion is managed through the iterative decomposition in Lemma \ref{lem08}.
\end{remark}
\begin{proof}
Note that
\begin{align*}
\nabla u_{m,i}&=\nabla u_{m,i-1}+\delta_{m+1}^{1/2}a_{i}\gamma_{i,2}\zeta_{i-1}\otimes\xi_{i}-\delta_{m+1}^{1/2}a_{i}\gamma_{i,1}\eta_{i-1}\otimes\xi_{i}\notag\\	
&\quad+\delta_{m+1}^{1/2}\lambda_{m,i}^{-1}\gamma_{i,1}\nabla(a_{i}\zeta_{i-1})+\delta_{m+1}^{1/2}\lambda_{m,i}^{-1}\gamma_{i,2}\nabla(a_{i}\eta_{i-1})\notag\\
&\quad+\delta_{m+1}\mathcal{F}_{i-1}\nabla w_{i-1}+\delta_{m+1}\nabla \mathcal{F}_{i-1}w_{i-1}\notag\\
&=:\nabla u_{m,i-1}+\mathcal{A}_{1}+\mathcal{A}_{2}+\mathcal{B}_{1}+\mathcal{B}_{2}+\mathcal{C}_{1}+\mathcal{C}_{2}.
\end{align*}	
Then we have
\begin{align*}
\nabla u_{m,i}^{t}\nabla u_{m,i}=&\nabla u_{m,i-1}^{t}\nabla u_{m,i-1}+2\mathrm{sym}\big(\nabla u_{m,i-1}^{t}(\mathcal{A}+\mathcal{B}+\mathcal{C})\big)\notag\\
&+(\mathcal{A}^{t}+\mathcal{B}^{t}+\mathcal{C}^{t})(\mathcal{A}+\mathcal{B}+\mathcal{C}),
\end{align*}	
where
\begin{align*}
\mathcal{A}=\mathcal{A}_{1}+\mathcal{A}_{2},\quad	\mathcal{B}=\mathcal{B}_{1}+\mathcal{B}_{2},\quad\mathcal{C}=\mathcal{C}_{1}+\mathcal{C}_{2}.
\end{align*}
From \eqref{AQ08}, we have
\begin{align*}
\nabla u_{m,i-1}^{t}\mathcal{A}=0,\quad\nabla u_{m,i-1}^{t}\mathcal{C}_{1}=\delta_{m+1}\nabla w_{i-1},\quad\|\nabla u_{m,i-1}^{t}\mathcal{C}_{2}\|_{0}\leq C\delta_{m+1}^{3/2}\Lambda^{-2/J},
\end{align*}
and
\begin{align*}
\nabla u_{m,i-1}^{t}\mathcal{B}=\frac{\delta_{m+1}^{1/2}a_{i}}{\lambda_{m,i}}(\gamma_{i,1}\nabla u_{m,i-1}^{t}\nabla\zeta_{i-1}+\gamma_{i,2}\nabla u_{m,i-1}^{t}\nabla\eta_{i-1}).	
\end{align*}		
For $i=1,...,n$, write
\begin{align*}
\mathcal{Q}_{i,1}=\frac{2a_{i}}{\delta_{m+1}^{1/2}\lambda_{m,i}}\mathrm{sym}(\nabla u_{m,i-1}^{t}\nabla\zeta_{i-1}),\;\, \mathcal{Q}_{i,2}=\frac{2a_{i}}{\delta_{m+1}^{1/2}\lambda_{m,i}}\mathrm{sym}(\nabla u_{m,i-1}^{t}\nabla\eta_{i-1}).
\end{align*}
It then follows from \eqref{AQ09}--\eqref{AQ10} that
\begin{align}\label{AQ12}
\sum^{2}_{j=1}\|\mathcal{Q}_{i,j}\|_{k}	\leq C\Lambda^{-1/J}\lambda_{m,i-1}^{k},\quad k\geq0.
\end{align}		
Then we obtain
\begin{align*}
&2\mathrm{sym}\big(\nabla u_{m,i-1}^{t}(\mathcal{A}+\mathcal{B}+\mathcal{C})\big)\notag\\
&=2\delta_{m+1}\mathrm{sym}(\nabla w_{i-1})+\delta_{m+1}\sum^{2}_{j=1}\gamma_{i,j}\mathcal{Q}_{i,j}+O(1) \delta_{m+1}^{3/2}\Lambda^{-2/J}.
\end{align*}	

It remains to calculate the term $(\mathcal{A}^{t}+\mathcal{B}^{t}+\mathcal{C}^{t})(\mathcal{A}+\mathcal{B}+\mathcal{C})$. Since $\zeta_{i-1}\cdot\eta_{i-1}=0$, we have
\begin{align*}
\mathcal{A}^{t}\mathcal{A}=\delta_{m+1}a_{i}^{2}\xi_{i}\otimes\xi_{i},	
\end{align*}	
and
\begin{align*}
\mathcal{A}^{T}\mathcal{B}=&\frac{\delta_{m+1}a_{i}}{\lambda_{m,i}}(\gamma_{i,2}\xi_{i}\otimes\zeta_{i-1}-\gamma_{i,1}\xi_{i}\otimes\eta_{i-1})(\gamma_{i,1}\nabla(a_{i}\zeta_{i-1})+\gamma_{i,2}\nabla(a_{i}\eta_{i-1}))\notag\\
=&\frac{\delta_{m+1}a_{i}\gamma_{i,1}\gamma_{i,2}}{\lambda_{m,i}}\big(\xi_{i}\otimes\zeta_{i-1}\nabla(a_{i}\zeta_{i-1})-\xi_{i}\otimes\eta_{i-1}\nabla(a_{i}\eta_{i-1})\big)\notag\\
&+\frac{\delta_{m+1}a_{i}}{\lambda_{m,i}}\big(\gamma_{i,2}^{2}\xi_{i}\otimes\zeta_{i-1}\nabla(a_{i}\eta_{i-1})-\gamma^{2}_{i,1}\xi_{i}\otimes\eta_{i-1}\nabla(a_{i}\zeta_{i-1})\big)\notag\\
=&\frac{\delta_{m+1}a_{i}^{2}\gamma_{i,1}\gamma_{i,2}}{\lambda_{m,i}}(\xi_{i}\otimes\zeta_{i-1}\nabla \zeta_{i-1}-\xi_{i}\otimes\eta_{i-1}\nabla\eta_{i-1})\notag\\
&+\frac{\delta_{m+1}a_{i}^{2}}{\lambda_{m,i}}(\gamma_{i,2}^{2}\xi_{i}\otimes\zeta_{i-1}\nabla\eta_{i-1}-\gamma_{i,1}^{2}\xi_{i}\otimes\eta_{i-1}\nabla\zeta_{i-1}),
\end{align*}	
which, together with \eqref{AQ10}, gives that
\begin{align*}
2\mathrm{sym}(\mathcal{A}^{T}\mathcal{B})=O(1)\delta_{m+1}^{3/2}\Lambda^{-1/J}.
\end{align*}	
In view of \eqref{AQ08}--\eqref{AQ10}, it follows from a direct computation that
\begin{align*}
\begin{cases}
2\mathrm{sym}(\mathcal{A}^{t}\mathcal{C})=O(1)\delta_{m+1}^{3/2}\Lambda^{-1/J},\quad2\mathrm{sym}(\mathcal{B}^{t}\mathcal{C})=O(1)\delta_{m+1}^{3/2}\Lambda^{-(i-1)/J},	\\
2\mathrm{sym}(\mathcal{C}^{t}\mathcal{C})=O(1)\delta_{m+1}^{2}\Lambda^{-2/J}.
\end{cases}
\end{align*}	
We proceed to compute the final term $\mathcal{B}^{t}\mathcal{B}$. Observe first that
 \begin{align*}
	\mathcal{B}^{t}\mathcal{B}&=\frac{\delta_{m+1}}{\lambda_{m,i}^{2}}\big(\gamma_{i,1}^{2}\nabla(a_{i}\zeta_{i-1})^{t}\nabla(a_{i}\zeta_{i-1})+\gamma_{i,2}^{2}\nabla(a_{i}\eta_{i-1})^{t}\nabla(a_{i}\eta_{i-1})\big)\notag\\
	&\quad+\frac{\delta_{m+1}\gamma_{i,1}\gamma_{i,2}}{\lambda_{m,i}^{2}}\big(\nabla(a_{i}\zeta_{i-1})^{t}\nabla(a_{i}\eta_{i-1})+\nabla(a_{i}\eta_{i-1})^{t}\nabla(a_{i}\zeta_{i-1})\big)\notag\\
	&=:\frac{\delta_{m+1}}{\lambda_{m,i}^{2}}I_{1}+\frac{\delta_{m+1}\gamma_{i,1}\gamma_{i,2}}{\lambda_{m,i}^{2}}I_{2}.
\end{align*}	
From the fact of $\zeta_{i-1}\cdot\eta_{i-1}=0$, we have
\begin{align*}
&\nabla(a_{i}\zeta_{i-1})^{t}\nabla(a_{i}\eta_{i-1})=(\nabla a_{i}\otimes\zeta_{i-1}+a_{i}\nabla\zeta_{i-1}^{t})(\eta_{i-1}\otimes\nabla a_{i}+a_{i}\nabla\eta_{i-1})\notag\\
&=a_{i}\nabla a_{i}\otimes\zeta_{i-1}\nabla\eta_{i-1}+a_{i}\nabla\zeta_{i-1}^{t}\eta_{i-1}\otimes\nabla a_{i}+a_{i}^{2}\nabla\zeta_{i-1}^{t}\nabla\eta_{i-1},	
\end{align*}	
 and
 \begin{align*}
 	&\nabla(a_{i}\zeta_{i-1})^{t}\nabla(a_{i}\zeta_{i-1})=(\nabla a_{i}\otimes\zeta_{i-1}+a_{i}\nabla\zeta_{i-1}^{t})(\zeta_{i-1}\otimes\nabla a_{i}+a_{i}\nabla\zeta_{i-1})\notag\\
 	&=\nabla a_{i}\otimes\nabla a_{i}+a_{i}\nabla a_{i}\otimes\zeta_{i-1}\nabla\zeta_{i-1}+a_{i}\nabla\zeta_{i-1}^{t}\zeta_{i-1}\otimes\nabla a_{i}+a_{i}^{2}\nabla\zeta_{i-1}^{t}\nabla\zeta_{i-1}.
 \end{align*}	
These two equations, in combination with \eqref{AQ09}--\eqref{AQ10}, show that
\begin{align*}
I_{1}=&\nabla a_{i}\otimes\nabla a_{i}+a_{i}\nabla a_{i}\otimes\big(\gamma_{i,1}^{2}\zeta_{i-1}\nabla\zeta_{i-1}+\gamma_{i,2}^{2}\eta_{i-1}\otimes\nabla\eta_{i-1}\big)\notag\\
&+a_{i}\big(\gamma_{i,1}^{2}\nabla\zeta_{i-1}^{t}\zeta_{i-1}+\gamma_{i,2}^{2}\nabla\eta_{i-1}^{t}\eta_{i-1}\big)\otimes\nabla a_{i}\notag\\
&+a_{i}^{2}\big(\gamma_{i,1}^{2}\nabla\zeta_{i-1}^{t}\nabla\zeta_{i-1}+\gamma_{i,2}^{2}\nabla\eta_{i-1}^{t}\nabla\eta_{i-1}\big)\notag\\
=&\nabla a_{i}\otimes\nabla a_{i}+O(1)\delta_{m+1}^{1/2}\lambda_{m,0}\lambda_{m,i-1},	
\end{align*}	
 and 
 \begin{align*}
 I_{2}=&a_{i}\nabla a_{i}\otimes (\zeta_{i-1}\nabla\eta_{i-1}+\eta_{i-1}\nabla\zeta_{i-1})+a_{i}(\nabla\zeta_{i-1}^{t}\eta_{i-1}+\nabla\eta_{i-1}^{t}\zeta_{i-1})\otimes\nabla a_{i}\notag\\	
 &+a_{i}^{2}(\nabla\zeta_{i-1}^{t}\nabla\eta_{i-1}+\nabla\eta_{i-1}^{t}\nabla\zeta_{i-1})=O(1)\delta_{m+1}^{1/2}\lambda_{m,0}\lambda_{m,i-1}.
 \end{align*}	
 Then we have
\begin{align*}
\mathcal{B}^{t}\mathcal{B}&=\frac{\delta_{m+1}}{\lambda_{m,i}^{2}}\nabla a_{i}\otimes\nabla a_{i}+O(1)\delta_{m+1}^{3/2}\Lambda^{-(i+1)/J},
\end{align*}	
and 
\begin{align*}
&(\mathcal{A}^{t}+\mathcal{B}^{t}+\mathcal{C}^{t})(\mathcal{A}+\mathcal{B}+\mathcal{C})\notag\\	
&=\delta_{m+1}a_{i}^{2}\xi_{i}\otimes\xi_{i}+\frac{\delta_{m+1}}{\lambda_{m,i}^{2}}\nabla a_{i}\otimes\nabla a_{i}+O(1)\delta_{m+1}^{3/2}\Lambda^{-1/J}.
\end{align*}	 
 Therefore, a consequence of these above facts leads to that
 \begin{align}\label{AQ15}
 \nabla u_{m,i}^{t}\nabla u_{m,i}=&\nabla u_{m,i-1}^{t}\nabla u_{m,i-1}+2\delta_{m+1}\mathrm{sym}(\nabla w_{i-1})+\delta_{m+1}\sum^{2}_{j=1}\gamma_{i,j}\mathcal{Q}_{i,j}\notag\\
 &+\delta_{m+1}a_{i}^{2}\xi_{i}\otimes\xi_{i}+\frac{\delta_{m+1}}{\lambda_{m,i}^{2}}\nabla a_{i}\otimes\nabla a_{i}+O(1)\delta_{m+1}^{3/2}\Lambda^{-1/J}.
 \end{align}	
In light of \eqref{AQ12}, we deduce from Lemma \ref{lem08} that for $i=1,...,n$ and $j=1,2,$ there exist $w_{i,j}\in C^{\infty}(\overline{\Omega}_{m},\mathbb{R}^{n})$, $\mathcal{G}_{i,j}\in C^{\infty}(\overline{\Omega}_{m},\mathrm{Sym}_{n})$, $\mathcal{R}_{i,j}\in C^{\infty}(\overline{\Omega}_{m},\mathrm{Sym}_{n})$, and $\bar{\gamma}_{i,j}\in C^{\infty}(\mathbb{S}^{1})$ with zero mean such that 
 \begin{align*}
 	\gamma_{i,j}(\lambda_{m,i} x\cdot\xi_{i})\mathcal{Q}_{i,j}=2\mathrm{sym}(\nabla w_{i,j})+\bar{\gamma}_{i,j}(\lambda_{m,i} x\cdot\xi_{i})(\lambda_{m,i-1}\lambda_{m,i}^{-1})^{J}\mathcal{G}_{i,j}+\mathcal{R}_{i,j},	
 \end{align*}		
where $\mathcal{R}_{i,j}\in\mathrm{span}\{\xi_{k}\otimes\xi_{k}:k=n+1,...,n_{\ast}\}$, and for $k\geq0$,
\begin{align*}
 		\|w_{i,j}\|_{k}\leq C\Lambda^{-1/J}\lambda_{m,i}^{k-1},\;\,\|\mathcal{G}_{i,j}\|_{k}\leq C\Lambda^{-1/J}\lambda_{m,i-1}^{k},\;\,\|\mathcal{R}_{i,j}\|_{k}\leq C\Lambda^{-1/J}\lambda_{m,i}^{k}.	
\end{align*}		
Take
\begin{align}\label{AQD001}
w_{i-1}=-\sum^{2}_{j=1}w_{i,j},\quad\mathcal{R}_{i}=\sum^{2}_{j=1}\mathcal{R}_{i,j},\quad i=1,...,n.	
\end{align}	
Then we arrive at
\begin{align*}
&2\delta_{m+1}\mathrm{sym}(\nabla w_{i-1})+\delta_{m+1}\sum^{2}_{j=1}\gamma_{i,j}\mathcal{Q}_{i,j}\notag\\
&=\delta_{m+1}\mathcal{R}_{i}+\delta_{m+1}\sum^{2}_{j=1}\bar{\gamma}_{i,j}(\lambda_{m,i} x\cdot\xi_{i})(\lambda_{m,i-1}\lambda_{m,i}^{-1})^{J}\mathcal{G}_{i,j}.
\end{align*}	
 Substituting this into \eqref{AQ15}, we obtain
  \begin{align}\label{AQ16}
 	\nabla u_{m,i}^{t}\nabla u_{m,i}=&\nabla u_{m,i-1}^{t}\nabla u_{m,i-1}+\delta_{m+1}a_{i}^{2}\xi_{i}\otimes\xi_{i}+\frac{\delta_{m+1}}{\lambda_{m,i}^{2}}\nabla a_{i}\otimes\nabla a_{i}\notag\\
 	&+\delta_{m+1}\mathcal{R}_{i}+O(1)\delta_{m+1}\Lambda^{-1/J}\big(\Lambda^{-1}+\delta_{m+1}^{1/2}\big).
 \end{align}	
By picking $\mathcal{R}=\sum^{n}_{i=1}\mathcal{R}_{i}$, we have from \eqref{AQ16} that Lemma \ref{lem006} holds.
 
 \end{proof}

\subsection{Constructions of the remaining perturbations}
We start by recalling the key components in Kuiper-type  corrugations (see Proposition 3.5 in \cite{DIS2018}) as follows.
\begin{lemma}\label{lem09}
There exist a small constant $\varepsilon_{\ast}>0$ and a function $\Gamma=(\Gamma_{1},\Gamma_{2})\in C^{\infty}([0,\varepsilon_{\ast}]\times\mathbb{R},\mathbb{R}^{2})$ such that for any $(s,t)\in[0,\varepsilon_{\ast}]\times\mathbb{R}$,
\begin{align*}
\Gamma(s,t)=\Gamma(s,t+2\pi),\quad(1+\partial_{t}\Gamma_{1})^{2}+(\partial_{t}\Gamma_{2})^{2}=1+s^{2},	
\end{align*}	
and, for $i\geq0,$
\begin{align*}
\begin{cases}
\|\partial_{t}^{i}\Gamma_{1}(s,\cdot)\|_{0}\leq C(i)s^{2},\quad \|\partial_{t}^{i}\Gamma_{2}(s,\cdot)\|_{0}\leq C(i)s,\quad\|\partial_{s}\partial_{t}^{i}\Gamma_{1}(s,\cdot)\|_{0}\leq C(i)s,\\
\|\partial_{s}\partial_{t}^{i}\Gamma_{2}(s,\cdot)\|_{0}\leq C(i),\quad\|\partial_{s}^{2}\partial_{t}^{i}\Gamma(s,\cdot)\|_{0}\leq C(i).
\end{cases}	
\end{align*}	

\end{lemma}
\begin{remark}
From Lemma 2 in \cite{CDS2012}, we see that $\Gamma=(\Gamma_{1},\Gamma_{2})$ is actually given by 
\begin{align*}
\begin{cases}	
\Gamma_{1}(s,t)=\int_{0}^{t}\big(\sqrt{1+s^{2}}\cos\big(\sqrt{f(s)}\sin \tau\big)-1\big)d\tau,\\
\Gamma_{2}(s,t)=\int_{0}^{t}\sqrt{1+s^{2}}\sin\big(\sqrt{f(s)}\sin \tau\big)d\tau,
\end{cases}	
\end{align*}	
where the implicit function $r=f(s)$ satisfies the following equation
\begin{align*}
F(s,r)=\dashint_{0}^{2\pi}\cos(r^{1/2}\sin t)dt-(1+s^{2})^{-1/2}=0.	
\end{align*}

\end{remark}

From Lemmas \ref{lem001} and \ref{lem006}, we see
\begin{align*}
\mathcal{R}=\sum^{n_{\ast}}_{i=n+1}L_{i}(\mathcal{R})\xi_{i}\otimes\xi_{i},
\end{align*}	  
where
\begin{align}\label{AQ18}
\|L_{i}(\mathcal{R})\|_{0}\leq C\|\mathcal{R}\|_{0}\leq C\Lambda^{-1/J}.
\end{align}	
Denote by $b_{i}$, $i=n+1,...,n_{\ast}$ the updated amplitudes as follows:
\begin{align}\label{AQ908}
b_{i}=\sqrt{a_{i}^{2}-L_{i}(\mathcal{R})}.	
\end{align}	 
From Lemma \ref{lem001} and \eqref{AQ18}, we have
\begin{align*}
a_{i}^{2}-L_{i}(\mathcal{R})\geq\sigma_{\ast}-C\Lambda^{-1/J}\geq\frac{\sigma_{\ast}}{2},\quad \text{as the parameter $a$ is sufficiently large},
\end{align*}	
which shows that $b_{i}$, $i=n+1,...,n_{\ast}$ are well-defined. Moreover, we have
\begin{align}\label{AQ19}
\|b_{i}\|_{k}\leq C(1+\Lambda^{-1/J}\lambda_{m,n}^{k}),\quad k\geq 0.	
\end{align}	

In the following we introduce the required perturbations by distinguishing between the cases of odd and even dimensions. 
\subsubsection{The odd-dimensional case}
Utilizing Lemma \ref{lem012}, we see that for every $j=1,...,\frac{n-1}{2}$, $u_{m,n+j-1}(\Omega_{m})$ has $n$ mutually orthogonal normal vectors $\{\tilde{\eta}_{n+j-1}^{(k)}\}_{k=nj+1}^{n(j+1)}$. For $k=nj+1,...,n(j+1)$, write 
\begin{align*}
\tilde{\zeta}_{n+j-1}^{(k)}=\nabla u_{m,n+j-1}(u_{m,n+j-1}^{\sharp}e)^{-1}\xi_{k},\quad \zeta_{n+j-1}^{(k)}=\frac{\tilde{\zeta}_{n+j-1}^{(k)}}{|\tilde{\zeta}^{(k)}_{n+j-1}|^{2}},	 
\end{align*}
and
\begin{align*}
\tilde{b}_{k}=\delta_{m+1}^{1/2}|\tilde{\zeta}_{n+j-1}^{(k)}|b_{k},\;\,\eta_{n+j-1}^{(k)}=\frac{\tilde{\eta}_{n+j-1}^{(k)}}{|\tilde{\eta}_{n+j-1}^{(k)}||\tilde{\zeta}_{n+j-1}^{(k)}|},
\end{align*}	
and, for $i=1,2,$
\begin{align*}
\begin{cases}	
\Gamma_{ik}=\Gamma_{i}(\tilde{b}_{k},\lambda_{m,n+j}x\cdot\xi_{k}),\;\,\partial_{s}\Gamma_{ik}=\partial_{s}\Gamma_{i}(\tilde{b}_{k},\lambda_{m,n+j}x\cdot\xi_{k}),\\
\partial_{t}\Gamma_{ik}=\partial_{t}\Gamma_{i}(\tilde{b}_{k},\lambda_{m,n+j}x\cdot\xi_{k}).
\end{cases}
\end{align*}	
Then for $j=1,...,\frac{n-1}{2}$, define
\begin{align}\label{W01}
u_{m,n+j}=u_{m,n+j-1}+\lambda_{m,n+j}^{-1}\sum^{n(j+1)}_{k=nj+1}\big(\Gamma_{1k}\zeta_{n+j-1}^{(k)}+\Gamma_{2k}\eta_{n+j-1}^{(k)}\big).	
\end{align}
Remark that the use of $n$ corrugation-type perturbations in \eqref{W01} corresponds to the dimension of the normal bundle, thereby fully employing the available codimension. This idea was already observed by K\"{a}llen \cite{K1978} in the context of Nash spirals rather than Kuiper corrugations. A direct computation gives that
\begin{align}\label{Q03}
	\|u_{m,\frac{3n-1}{2}}-u_{m,n}\|_{k}\leq C_{3}(n)\delta_ {m+1}^{1/2}\lambda_{m,n+1}^{k-1},\quad k=0,1.
\end{align}
By the same argument as in \eqref{AQ09}, we have
\begin{align}\label{AQ23}
\|u_{m,n+j}\|_{k}\leq C\big(1+\delta_{m+1}^{1/2}\lambda_{m,n+j}^{k-1}\big),\quad 1\leq j\leq \frac{n-1}{2},\;k\geq0.
\end{align}	

The main result of this subsection is now listed as follows.
\begin{lemma}\label{lem007}
Let $n\geq3$ be an odd number. Then we have
\begin{align*}
	u_{m,\frac{3n-1}{2}}^{\sharp}e=&u_{m,n}^{\sharp}e+\delta_{m+1}\sum^{n_{\ast}}_{i=n+1}b_{i}^{2}\xi_{i}\otimes\xi_{i}+O(1)\delta_{m+1}(\Lambda^{-1}+\delta_{m+1}),
\end{align*}		
where $b_{i}$, $i=n+1,...,n_{\ast}$ are defined by \eqref{AQ908}.		
\end{lemma}

\begin{proof}
For $j=1,...,\frac{n-1}{2}$, we have
\begin{align*}
\nabla	u_{m,n+j}&=\nabla u_{m,n+j-1}+\sum^{n(j+1)}_{k=nj+1}\big(\partial_{t}\Gamma_{1k}\zeta_{n+j-1}^{(k)}\otimes\xi_{k}+\partial_{t}\Gamma_{2k}\eta_{n+j-1}^{(k)}\otimes\xi_{k}\big)\notag\\
&\quad+\sum^{n(j+1)}_{k=nj+1}\lambda_{m,n+j}^{-1}\big(\Gamma_{1k}\nabla\zeta_{n+j-1}^{(k)}+\Gamma_{2k}\nabla\eta_{n+j-1}^{(k)}\big)\notag\\
&\quad+\sum^{n(j+1)}_{k=nj+1}\lambda_{m,n+j}^{-1}\big(\partial_{s}\Gamma_{1k}\zeta_{n+j-1}^{(k)}\otimes\nabla\tilde{b}_{k}+\partial_{s}\Gamma_{2k}\eta_{n+j-1}^{(k)}\otimes\nabla\tilde{b}_{k}\big)\notag\\
&=:\nabla u_{m,n+j-1}+\sum^{n(j+1)}_{k=nj+1}(\mathcal{A}_{k}+\mathcal{B}_{k}+\mathcal{C}_{k}).
\end{align*}	
This leads to that
\begin{align}\label{AQ28}
 u_{m,n+j}^{\sharp}e=&u_{m,n+j-1}^{\sharp}e+\sum^{n(j+1)}_{k=nj+1}2\mathrm{sym}\big(\nabla u_{m,n+j-1}^{t}(\mathcal{A}_{k}+\mathcal{B}_{k}+\mathcal{C}_{k})\big)\notag\\
	&+\sum^{n(j+1)}_{k,l=nj+1}(\mathcal{A}_{k}^{t}+\mathcal{B}_{k}^{t}+\mathcal{C}_{k}^{t})(\mathcal{A}_{l}+\mathcal{B}_{l}+\mathcal{C}_{l}).
\end{align}	
Since
\begin{align*}
\nabla u_{m,n+j-1}^{t}\eta_{n+j-1}^{(k)}=0,\quad \nabla u_{m,n+j-1}^{t}\zeta_{n+j-1}^{(k)}=\frac{\xi_{k}}{|\tilde{\zeta}_{n+j-1}^{(k)}|},
\end{align*}		
we deduce
\begin{align*}
\nabla u_{m,n+j-1}^{t}\mathcal{A}_{k}=&\partial_{t}\Gamma_{1k}\nabla u_{m,n+j-1}^{t}\zeta_{n+j-1}^{(k)}\otimes\xi_{k}+\partial_{t}\Gamma_{2k}\nabla u_{m,n+j-1}^{t}\eta_{n+j-1}^{(k)}\otimes\xi_{k}	\notag\\
=&\frac{\partial_{t}\Gamma_{1k}}{|\tilde{\zeta}_{n+j-1}^{(k)}|}\xi_{k}\otimes\xi_{k},
\end{align*}		
which reads that 
$$2\mathrm{sym}(\nabla u_{m,n+j-1}^{t}\mathcal{A}_{k})=\frac{2\partial_{t}\Gamma_{1k}}{|\tilde{\zeta}_{n+j-1}^{(k)}|}\xi_{k}\otimes\xi_{k}.$$	
Making use of Lemma \ref{lem09} and \eqref{AQ19}--\eqref{AQ23}, we deduce
\begin{align*}
\|\nabla u_{m,n+j-1}^{t}\mathcal{B}_{k}\|_{0}=&\lambda_{m,n+j}^{-1}\big\|\Gamma_{1k}\nabla u_{m,n+j-1}^{t}\nabla\zeta_{n+j-1}^{(k)}+\Gamma_{2k}\nabla u_{m,n+j-1}^{t}\nabla\eta_{n+j-1}^{(k)}\big\|_{0}\notag\\
\leq&C\lambda_{m,n+j}^{-1}\|\tilde{b}_{k}\|_{0}\|u_{m,n+j-1}\|_{2}\leq C\delta_{m+1}\lambda^{-1}_{m,n+j}\lambda_{m,n+j-1},
\end{align*}	
which yields that
\begin{align*}
2\mathrm{sym}(\nabla u_{m,n+j-1}^{t}\mathcal{B}_{k})=O(1)\delta_{m+1}\Lambda^{-1}.
\end{align*}	
Similarly, we obtain
\begin{align*}
&\nabla u_{m,n+j-1}^{t}\mathcal{C}_{k}\notag\\
&=\lambda_{m,n+j}^{-1}\big(\partial_{s}\Gamma_{1k}\nabla u_{m,n+j-1}^{t}\zeta_{n+j-1}^{(k)}\otimes\nabla\tilde{b}_{k}+\partial_{s}\Gamma_{2k}\nabla u_{m,n+j-1}^{t}\eta_{n+j-1}^{(k)}\otimes\nabla\tilde{b}_{k}\big)\notag\\
&=\frac{\partial_{s}\Gamma_{1k}\xi_{k}\otimes\nabla\tilde{b}_{k}}{\lambda_{m,n+j}|\tilde{\zeta}_{n+j-1}^{(k)}|}=O(1)\frac{\|\tilde{b}_{k}\|_{0}\|\nabla\tilde{b}_{k}\|_{0}}{\lambda_{m,n+j}}=O(1)\frac{\delta_{m+1}\lambda_{m,n}}{\Lambda^{1/J}\lambda_{m,n+j}},
\end{align*}	
which shows that
\begin{align*}
2\mathrm{sym}(\nabla u_{m,n+j-1}^{t}\mathcal{C}_{k})=O(1)\Lambda^{-j-1/J}\delta_{m+1}.	
\end{align*}	
A consequence of these above facts gives that
\begin{align*}
&\sum^{n(j+1)}_{k=nj+1}2\mathrm{sym}\big(\nabla u_{m,n+j-1}^{t}(\mathcal{A}_{k}+\mathcal{B}_{k}+\mathcal{C}_{k})\big)\notag\\	
&=\sum^{n(j+1)}_{k=nj+1}\frac{2\partial_{t}\Gamma_{1k}}{|\tilde{\zeta}_{n+j-1}^{(k)}|}\xi_{k}\otimes\xi_{k}+O(1)\delta_{m+1}\Lambda^{-1}.
\end{align*}		
	
We now continue by computing the remaining term	in \eqref{AQ28}. First, we further split it into two parts as follows:
\begin{align*}
\begin{cases}
I_{1}=\sum^{n(j+1)}_{k=nj+1}(\mathcal{A}_{k}^{t}+\mathcal{B}_{k}^{t}+\mathcal{C}_{k}^{t})(\mathcal{A}_{k}+\mathcal{B}_{k}+\mathcal{C}_{k}),\\
I_{2}=\sum^{n(j+1)}_{k,l=nj+1,k\neq l}(\mathcal{A}_{k}^{t}+\mathcal{B}_{k}^{t}+\mathcal{C}_{k}^{t})(\mathcal{A}_{l}+\mathcal{B}_{l}+\mathcal{C}_{l}).
\end{cases}	
\end{align*}	
Since $\zeta_{n+j-1}^{(k)}\cdot\eta_{n+j-1}^{(k)}=0$, then
\begin{align*}
\mathcal{A}^{t}_{k}\mathcal{A}_{k}=&\big((\partial_{t}\Gamma_{1k})^{2}|\zeta_{n+j-1}^{(k)}|^{2}+(\partial_{t}\Gamma_{2k})^{2}|\eta_{n+j-1}^{(k)}|^{2}\big)\xi_{k}\otimes\xi_{k}\notag\\
=&\frac{(\partial_{t}\Gamma_{1k})^{2}+(\partial_{t}\Gamma_{2k})^{2}}{|\tilde{\zeta}_{n+j-1}^{(k)}|^{2}}\xi_{k}\otimes\xi_{k}.	
\end{align*}	
In light of Lemma \ref{lem09} and \eqref{AQ19}--\eqref{AQ23}, it follows from a direct computation that
\begin{align}\label{AQ29}
\begin{cases}		
\mathcal{A}^{t}_{k}\mathcal{B}_{k}=O(1)\lambda_{m,n+j}^{-1}\|\tilde{b}_{k}\|_{0}^{2}\|u_{m,n+j-1}\|_{2}=O(1)\Lambda^{-1}\delta_{m+1}^{3/2},\\
\mathcal{A}^{t}_{k}\mathcal{C}_{k}=O(1)\lambda_{m,n+j}^{-1}\|\tilde{b}_{k}\|_{0}\|\nabla\tilde{b}_{k}\|_{0}=O(1)\Lambda^{-1/J-j}\delta_{m+1},\\
\mathcal{B}^{t}_{k}\mathcal{B}_{k}=O(1)\lambda_{m,n+j}^{-2}(\|\tilde{b}_{k}\|_{0}\|u_{m,n+j-1}\|_{2})^{2}=O(1)\Lambda^{-2}\delta_{m+1}^{2},\\
\mathcal{B}^{t}_{k}\mathcal{C}_{k}=O(1)\lambda_{m,n+j}^{-2}\|\tilde{b}_{k}\|_{0}\|\nabla\tilde{b}_{k}\|_{0}\|u_{m,n+j-1}\|_{2}=O(1)\Lambda^{-1/J-j-1}\delta_{m+1}^{3/2},\\
\mathcal{C}^{t}_{k}\mathcal{C}_{k}=O(1)\lambda_{m,n+j}^{-2}\|\nabla\tilde{b}_{k}\|_{0}^{2}=O(1)\Lambda^{-2/J-2j}\delta_{m+1}.
\end{cases}
\end{align}		
Therefore, we have
\begin{align*}
I_{1}=\sum^{n(j+1)}_{k=nj+1}\frac{(\partial_{t}\Gamma_{1k})^{2}+(\partial_{t}\Gamma_{2k})^{2}}{|\tilde{\zeta}_{n+j-1}^{(k)}|^{2}}\xi_{k}\otimes\xi_{k}+O(1)\Lambda^{-1-1/J}\delta_{m+1}.
\end{align*}		
	
With regard to the second term $I_{2}$, using the fact of $\eta_{n+j-1}^{(k)}\cdot\zeta_{n+j-1}^{(l)}=\eta_{n+j-1}^{(k)}\cdot\eta_{n+j-1}^{(l)}=0$ with $k,l=nj+1,\,k\neq l$, we have from Lemma \ref{lem09} and \eqref{AQ19}--\eqref{AQ23} again that	
\begin{align*}
\mathcal{A}_{k}^{t}\mathcal{A}_{l}=	\partial_{t}\Gamma_{1k}\partial_{t}\Gamma_{1l}\zeta_{n+j-1}^{(k)}\cdot\zeta_{n+j-1}^{(l)}\xi_{k}\otimes\xi_{k}=O(1)\|\tilde{b}_{k}\|_{0}^{4}=O(1)\delta_{m+1}^{2}.
\end{align*}		
In exactly the same way as in \eqref{AQ29}, we obtain that for $k,l=nj+1,\,k\neq l$,
\begin{align*}
	\begin{cases}		
		\mathcal{A}^{t}_{k}\mathcal{B}_{l}=O(1)\Lambda^{-1}\delta_{m+1}^{3/2},\quad\mathcal{A}^{t}_{k}\mathcal{C}_{l}=O(1)\Lambda^{-1/J-j}\delta_{m+1},\\
		\mathcal{B}^{t}_{k}\mathcal{B}_{l}=O(1)\Lambda^{-2}\delta_{m+1}^{2},\quad\mathcal{B}^{t}_{k}\mathcal{C}_{l}=O(1)\Lambda^{-1/J-j-1}\delta_{m+1}^{3/2},\\
		\mathcal{C}^{t}_{k}\mathcal{C}_{l}=O(1)\Lambda^{-2/J-2j}\delta_{m+1}.
	\end{cases}
\end{align*}			
Hence we have
\begin{align*}
I_{2}=O(1)\delta_{m+1}(\Lambda^{-1-1/J}+\delta_{m+1}).
\end{align*}		
Consequently, combining these above facts, we deduce from Lemma \ref{lem09} that	
\begin{align*}
u_{m,n+j}^{\sharp}e=&u_{m,n+j-1}^{\sharp}e+\sum^{n(j+1)}_{k=nj+1}\frac{2\partial_{t}\Gamma_{1k}+(\partial_{t}\Gamma_{1k})^{2}+(\partial_{t}\Gamma_{2k})^{2}}{|\tilde{\zeta}_{n+j-1}^{(k)}|^{2}}\xi_{k}\otimes\xi_{k}\notag\\
&+O(1)\delta_{m+1}(\Lambda^{-1}+\delta_{m+1})\notag\\
=&u_{m,n+j-1}^{\sharp}e+\sum^{n(j+1)}_{k=nj+1}\delta_{m+1}b_{k}^{2}\xi_{k}\otimes\xi_{k}+O(1)\delta_{m+1}(\Lambda^{-1}+\delta_{m+1}),
\end{align*}	
which implies that Lemma \ref{lem007} holds.	
		
\end{proof}	

\subsubsection{The even-dimensional case}\label{SU03}

In the following we adopt $\frac{n}{2}$ Nash-spirals in the first step and then Kuiper-corrugations for $\frac{n-2}{2}$ times to cancel $\frac{n}{2}$ and $\frac{n(n-2)}{2}$ primitive metrics in \eqref{AQ06}, respectively.

In view of Lemma \ref{lem012}, we know that for any $j=1,...,\frac{n}{2}$, $u_{m,n+j-1}(\Omega_{m})$ possesses $n$ mutually orthogonal normal vectors, respectively, written as $\{\zeta_{n}^{(k)},\eta_{n}^{(k)}\}_{k=1}^{\frac{n}{2}}$ and $\{\tilde{\eta}_{n+j-1}^{(k)}\}_{k=n(j-2)+\frac{n}{2}+1}^{n(j-1)+\frac{n}{2}}$, $j=2,...,\frac{n}{2}$. We now divide into two cases and introduce the required Nash-spirals and Kuiper-corrugations as follows. On one hand, for $j=1,\,k=n+1,...,\frac{3n}{2}$, define
\begin{align*}
	u_{m,n+1}=u_{m,n}+\lambda_{m,n+1}^{-1}\sum_{k=n+1}^{\frac{3n}{2}}\tilde{b}_{1k}\big(\gamma_{k,1}\zeta_{n}^{(k)}+\gamma_{k,2}\eta_{n}^{(k)}\big),	
\end{align*}	
where
\begin{align*}
\tilde{b}_{1k}=\delta_{m+1}^{1/2}b_{k},\quad	\gamma_{k,1}=\sin(\lambda_{m,n+1}x\cdot\xi_{k}), \quad\gamma_{k,2}=\cos(\lambda_{m,n+1}x\cdot\xi_{k}).
\end{align*}

On the other hand, for $j=2,...,\frac{n}{2},\,k=n(j-2)+\frac{3n}{2}+1,...,n(j-1)+\frac{3n}{2}$, denote
\begin{align*}
	\tilde{\zeta}_{n+j-1}^{(k)}=\nabla u_{m,n+j-1}(u_{m,n+j-1}^{\sharp}e)^{-1}\xi_{k},\quad \zeta_{n+j-1}^{(k)}=\frac{\tilde{\zeta}_{n+j-1}^{(k)}}{|\tilde{\zeta}^{(k)}_{n+j-1}|^{2}},
\end{align*}
and, for $i=1,2,$
\begin{align*}
	\tilde{b}_{2k}=\delta_{m+1}^{1/2}|\tilde{\zeta}_{n+j-1}^{(k)}|b_{k},	\;	\eta_{n+j-1}^{(k)}=\frac{\tilde{\eta}_{n+j-1}^{(k)}}{|\tilde{\eta}_{n+j-1}^{(k)}||\tilde{\zeta}_{n+j-k}^{(k)}|},\;
	\Gamma_{ik}=\Gamma_{i}(\tilde{b}_{2k},\lambda_{m,n+j}x\cdot\xi_{k}).
\end{align*}	
Then for $j=2,...,\frac{n}{2}$, define
\begin{align}\label{W02}
	u_{m,n+j}=u_{m,n+j-1}+\lambda_{m,n+j}^{-1}\sum^{n(j-1)+\frac{3n}{2}}_{k=n(j-2)+\frac{3n}{2}+1}\big(\Gamma_{1k}\zeta_{n+j-1}^{(k)}+\Gamma_{2k}\eta_{n+j-1}^{(k)}\big).	
\end{align}	
Hence we have
\begin{align}\label{Q05}
	\|u_{m,\frac{3n}{2}}-u_{m,n}\|_{k}\leq C_{4}(n)\delta_ {m+1}^{1/2}\lambda_{m,n+1}^{k-1},\quad k=0,1.
\end{align}
Similar to \eqref{AQ09}, for $j=1,...,\frac{n}{2}$, we have
\begin{align}\label{AQ369}
	\|u_{m,n+j}\|_{k}\leq C\big(1+\delta_{m+1}^{1/2}\lambda_{m,n+j}^{k-1}\big),\quad k\geq0. 
\end{align}	
In light of \eqref{AQ369} and applying the proof of Lemma \ref{lem007} with a slight modification, we directly have the following result.
\begin{corollary}\label{coro033}
Let $n\geq4$ be an even number. Then there holds
	\begin{align*}
		u_{m,\frac{3n}{2}}^{\sharp}e=&u_{m,n}^{\sharp}e+\delta_{m+1}\sum^{n_{\ast}}_{k=n+1}b_{k}^{2}\xi_{k}\otimes\xi_{k}+O(1)\delta_{m+1}(\Lambda^{-1}+\delta_{m+1}),
	\end{align*}		
	where $b_{k}$, $k=n+1,...,n_{\ast}$ are given by \eqref{AQ908}.	
\end{corollary}

We are now ready to give the proof of Proposition \ref{pro01}. 
\begin{proof}[Proof of Proposition \ref{pro01}]
Pick 
\begin{align*}
u_{m+1}:=
\begin{cases}
u_{m,\frac{3n-1}{2}},&\text{if $n$ is odd,}\\
u_{m,\frac{3n}{2}},&\text{if $n$ is even,}
\end{cases}		
\end{align*}		
where $u_{m,\frac{3n-1}{2}}$ and $u_{m,\frac{3n}{2}}$ are, respectively, defined by \eqref{W01} and \eqref{W02}. Making use of \eqref{AQ23} and \eqref{AQ369}, we deduce
\begin{align*}
\|\nabla^{2}u_{m+1}\|_{0}\leq&C
\begin{cases}
\delta_{m+1}^{1/2}\lambda_{m,\frac{3n-1}{2}},&\text{if $n$ is odd,}\\
\delta_{m+1}^{1/2}\lambda_{m,\frac{3n}{2}},&\text{if $n$ is even}.		
\end{cases}
\end{align*}	
Recalling \eqref{W018}--\eqref{W86} and in order to make $\|\nabla^{2}u_{m+1}\|_{0}\leq \delta_{m+1}^{1/2}\lambda_{m+1},$ we need to guarantee that
\begin{align*}
\begin{cases}
C\delta^{1/2}_{m}\lambda_{m}a^{\alpha}\big(\frac{\delta_{m+1}a^{\alpha}}{\delta_{m+2}}\big)^{\frac{J(n-1)+2n}{2J}}\leq\delta_{m+1}^{1/2}\lambda_{m+1},&\text{if $n$ is odd,}\\
C\delta^{1/2}_{m}\lambda_{m}a^{\alpha}\big(\frac{\delta_{m+1}a^{\alpha}}{\delta_{m+2}}\big)^{\frac{n(J+2)}{2J}}\leq\delta_{m+1}^{1/2}\lambda_{m+1},&\text{if $n$ is even.}
\end{cases}	
\end{align*}	
Then it suffices to require that 

$(i)$ if $n$ is odd,
\begin{align}\label{Z08}
\begin{cases}
1-\vartheta/b+(2\vartheta(b-1)+b\alpha)(J(n-1)+2n)/(2J)<b-\vartheta,\\
\alpha(J(n+1)+2n)/(2J)<b(b-1)[1-\vartheta(n+2n/J-(b-1)/b)];
\end{cases}
\end{align}	

$(ii)$ if $n$ is even,
\begin{align}\label{Z18}
	\begin{cases}
		1-\vartheta/b+n(2\vartheta(b-1)+b\alpha)(J+2)/(2J)<b-\vartheta,\\
		\alpha(J(n+2)+2n)/(2J)<b(b-1)[1-\vartheta(n+1+2n/J-(b-1)/b)].
	\end{cases}
\end{align}	
For the purpose of ensuring that the facts in Remark \ref{REM01} and \eqref{Z08}--\eqref{Z18} hold, we fix the values of $b,\vartheta,\alpha,J$ in \eqref{W90} and \eqref{W901}. Indeed, on one hand, using \eqref{W90} and \eqref{W901}, we have from a direct computation that for any $n\geq3$,
\begin{align}\label{W006}
\theta<\frac{\vartheta}{b}<\frac{\theta+\varepsilon}{1+\beta},\quad\beta=
\begin{cases}
	\frac{J(n-1)(b-1)+2bn}{Jn},&\text{if $n$ is odd},\\
	\frac{Jn(b-1)+2bn}{J(n+1)},&\text{if $n$ is even},
\end{cases}			
\end{align}	
which implies that the facts in Remark \ref{REM01} and the first inequalities of \eqref{Z08}--\eqref{Z18} hold. On the other hand, we have from \eqref{W90} and \eqref{W901} again that 

$(1)$ if $n$ is odd, 
\begin{align*}
&1-\vartheta\Big(n+\frac{2n}{J}-\frac{b-1}{b}\Big)\geq 1-\vartheta\Big(n+\frac{(n-1)^{2}\varepsilon}{2b}\Big)\notag\\
&\geq1-(1-(n-1)\varepsilon)\Big(b+\frac{(n-1)^{2}\varepsilon}{2n}\Big)\geq\frac{(n-1)\varepsilon}{2n},
\end{align*}
which gives that
\begin{align*}
b(b-1)\bigg[1-\vartheta\Big(n+\frac{2n}{J}-\frac{b-1}{b}\Big)\bigg]>\frac{(n-1)^{2}\varepsilon^{2}}{4n};	
\end{align*}		

$(2)$ if $n$ is even,
 \begin{align*}
 	&1-\vartheta\Big(n+1+\frac{2n}{J}-\frac{b-1}{b}\Big)\geq 1-\vartheta\Big(n+1+\frac{(n-1)^{2}\varepsilon}{2b}\Big)\notag\\
 	&\geq1-(1-(n-1)\varepsilon)\Big(b+\frac{(n-1)^{2}\varepsilon}{2(n+1)}\Big)\geq\frac{\varepsilon(n-1)}{n+1},
 \end{align*}	
which reads that
\begin{align*}
b(b-1)\bigg[1-\vartheta\Big(n+1+\frac{2n}{J}-\frac{b-1}{b}\Big)\bigg]>\frac{(n-1)^{2}\varepsilon^{2}}{2(n+1)}.	
\end{align*}	 
If $0<\varepsilon<2/n^{2}$, we have $J\geq2n$. Then we obtain
\begin{align*}
\frac{Jk+2n}{2J}\leq\frac{k+1}{2},\quad k=n+1,n+2.
\end{align*}
So in order to let the second inequalities in \eqref{Z08}--\eqref{Z18} hold, it is sufficient to take $\alpha=\frac{\varepsilon^{2}}{10}$ such that
\begin{align*}
\alpha\leq\frac{(n-1)^{2}\varepsilon^{2}}{2n(n+2)},\quad\forall\,n\geq3.
\end{align*}	
This, together with \eqref{W006}, confirms the validity of the statements in Remark \ref{REM01} and \eqref{Z08}--\eqref{Z18}. In addition, in view of \eqref{W90} and \eqref{W901}, the inequality of $\delta_{m+1}^{1/2}\leq\Lambda^{-1}$ holds up to the constant $\delta_{\ast}^{1/2}$ provided that
\begin{align*}
-(n-1)^{2}\varepsilon^{3}-\frac{9n-11}{5(n+1)}\varepsilon^{2}+\frac{3n+1}{n+1}\varepsilon-\frac{2}{n+1}\leq0.	
\end{align*}	
This inequality is valid, if $0<\varepsilon<2/(3n+1)$. Set 
\begin{align}\label{AQ56}
	C_{\ast}:=
	\begin{cases}
		\max\{C_{1}(n),C_{2}(n),C_{3}(n)\},&\text{if $n$ is odd,}\\
		\max\{C_{1}(n),C_{2}(n),C_{4}(n)\},&\text{if $n$ is even},	
	\end{cases}	
\end{align}
where $C_{i}(n),i=1,2,3,4$ are given by \eqref{Q01}, \eqref{Q02}, \eqref{Q03} and \eqref{Q05}, respectively. Consequently, a combination of \eqref{Q01}, \eqref{Q02}, \eqref{Q03}--\eqref{AQ23}, \eqref{Q05}--\eqref{AQ369}, \eqref{AQ56}, Lemmas \ref{lem006} and \ref{lem007}, and Corollary \ref{coro033} shows that Proposition \ref{pro01} holds, as the parameter $a$ is taken large enough.

\end{proof}

\subsection{The proof of Theorem \ref{Main01}}
To complete the proof of Theorem \ref{Main01}, we see from Proposition \ref{pro01} that it remains to construct a short immersion $u_{0}\in C^{\infty}(\overline{\Omega}_{0},\mathbb{R}^{2n})$ subject to
\begin{align}\label{AQ65}
\begin{cases}
\|u_{0}-\bar{u}\|_{0}\leq2^{-1}\varepsilon,\quad\|u_{0}\|_{1}\leq\|\bar{u}\|_{1}+\mathcal{G}_{\ast},\quad\|\nabla(u_{0}-\bar{u})\|_{0}\leq\mathcal{K}_{\ast},\\
\|g_{0}-u^{\sharp}_{0}e\|_{0}\leq\sigma_{0}\delta_{1},\quad\|\nabla^{2} u_{0}\|_{0}\leq\delta_{0}^{1/2}\lambda_{0},
\end{cases}	
\end{align}
where $\mathcal{G}_{\ast},\mathcal{K}_{\ast}$ are given by \eqref{AQ68} and \eqref{AQ69}, respectively. For $\beta\in(0,1]$, define	
\begin{align*}
	\Omega^{\ast}_{\beta}=\{x\in\mathbb{R}^{n}:\mathrm{dist}(x,\Omega_{0})<\beta\},	
\end{align*}		
where $\Omega_{0}$ is defined by \eqref{W009}. In light of the classical extension arguments, we suppose without loss of generality that  $\bar{u}:(\Omega_{1}^{\ast},g)\hookrightarrow\mathbb{R}^{2n}$ is a short $ C^{1,\varepsilon}$ immersion. For a matrix $A\in\mathbb{R}^{n\times n}$, denote
\begin{align*}
\lambda_{\min}(A)=\min_{1\leq i\leq n}\lambda_{i}(A),\quad	\lambda_{\max}(A)=\max_{1\leq i\leq n}\lambda_{i}(A),
\end{align*}	
where $\{\lambda_{i}(A)\}_{i=1}^{n}$ are the eigenvalues of $A$. Note that
\begin{align*}
g-\bar{u}^{\sharp}e\geq\inf_{x\in\overline{\Omega_{1}^{\ast}}}\lambda_{\min}(g-\bar{u}^{\sharp}e)e,\quad h_{\ast}\leq \lambda_{\max}(h_{\ast}) e,
\end{align*}
which, together with the compactness of the region $\overline{\Omega_{1}^{\ast}}$, reads that 
\begin{align}\label{W08}
g-\bar{u}^{\sharp}e\geq5\delta_{\ast}h_{\ast},\quad \delta_{\ast}:=\frac{\inf_{x\in\overline{\Omega_{1}^{\ast}}}\lambda_{\min}(g-\bar{u}^{\sharp}e)}{5\lambda_{\max}(h_{\ast})},	
\end{align}	
where $h_{\ast}$ is defined by \eqref{M01}. Pick
\begin{align*}
\bar{\ell}=\frac{\sigma_{0}\delta_{1}}{2\|g\|_{1}},
\end{align*}	 
where $\sigma_{0},\delta_{1}$ are, respectively, given by Lemma \ref{lem002} and \eqref{M06}. Note that
\begin{align*}
g\ast\varphi_{\bar{\ell}}-(\bar{u}\ast\varphi_{\bar{\ell}})^{\sharp}{e}=&g-\bar{u}^{\sharp}e+g\ast\varphi_{\bar{\ell}}-g+\bar{u}^{\sharp}e-(\bar{u}^{\sharp}e)\ast\varphi_{\bar{\ell}}\notag\\
&+(\bar{u}^{\sharp}e)\ast\varphi_{\bar{\ell}}-(\bar{u}\ast\varphi_{\bar{\ell}})^{\sharp}{e},
\end{align*}	
where $\varphi_{\bar{\ell}}$ represents the standard mollifying kernel on $\mathbb{R}^{n}$. Then we have from Lemma \ref{LEM08} that
\begin{align*}
\begin{cases}
\|g-g\ast\varphi_{\bar{\ell}}\|_{0}\leq \|g\|_{1}\bar{\ell},\quad\|\bar{u}^{\sharp}e-(\bar{u}^{\sharp}e)\ast\varphi_{\bar{\ell}}\|_{0}\leq \bar{c}_{1}\|\bar{u}\|_{1+\varepsilon}\bar{\ell}^{\varepsilon},	\\
\|(\bar{u}^{\sharp}e)\ast\varphi_{\bar{\ell}}-(\bar{u}\ast\varphi_{\bar{\ell}})^{\sharp}{e}\|_{0}\leq \bar{c}_{2}\bar{\ell}^{2\varepsilon}\|\bar{u}\|_{1+\varepsilon}^{2},
\end{cases}
\end{align*}	
where $\bar{c}_{i}=\bar{c}_{i}(n)$, $i=1,2.$ Then if the parameter $a$ is sufficiently large, 
we have
\begin{align*}
\bar{\ell}\leq\min\bigg\{&\frac{\delta_{\ast}\lambda_{\min}(h_{\ast})}{\|g\|_{1}},\bigg(\frac{\delta_{\ast}\lambda_{\min}(h_{\ast})}{\bar{c}_{1}\|\bar{u}\|_{1+\varepsilon}}\bigg)^{\frac{1}{\varepsilon}},\bigg(\frac{\delta_{\ast}\lambda_{\min}(h_{\ast})}{\bar{c}_{2}\|\bar{u}\|^{2}_{1+\varepsilon}}\bigg)^{\frac{1}{2\varepsilon}}\bigg\},	
\end{align*}	
which leads to that
\begin{align*}
g\ast\varphi_{\bar{\ell}}-(\bar{u}\ast\varphi_{\bar{\ell}})^{\sharp}{e}\geq	2\delta_{\ast}h_{\ast}\geq	2\delta_{1}h_{\ast}.
\end{align*}	
Therefore, a consequence of Lemma \ref{lem001}, finite covering theorem, Caratheodory's theorem and partition of unity shows that there exist some constant $N_{\ast}\geq n_{\ast}$, a family of unit vectors $\{\nu_{i}\}_{i=1}^{N_{\ast}}\subset\mathbb{S}^{n}$ and related amplitudes $\{\bar{a}_{i}\}_{i=1}^{N_{\ast}}\subset C^{\infty}(\mathrm{Sym}_{n},\mathbb{R}_{+})$ such that
\begin{align}\label{W16}
g\ast\varphi_{\bar{\ell}}-(\bar{u}\ast\varphi_{\bar{\ell}})^{\sharp}{e}-\delta_{1}h_{\ast}=\sum^{N_{\ast}}_{i=1}\bar{a}_{i}^{2}\nu_{i}\otimes\nu_{i},	
\end{align}	
where $\mathbb{R}_{+}:=\{x\in\mathbb{R}:x>0\}$ and $\bar{a}_{i}$, $i=1,...,N_{\ast}$ satisfy that for $k\geq1$,
\begin{align*}
\|\bar{a}_{i}\|_{k}\leq& C\|g\ast\varphi_{\bar{\ell}}-(\bar{u}\ast\varphi_{\bar{\ell}})^{\sharp}{e}\|_{k}\leq C\big(\|(g-\bar{u}^{\sharp}e)\ast\varphi_{\bar{\ell}}\|_{k}+\|(\bar{u}^{\sharp}e)\ast\varphi_{\bar{\ell}}-(\bar{u}\ast\varphi_{\bar{\ell}})^{\sharp}{e}\|_{k}\big)\notag\\
\leq& C\big(\|g-\bar{u}^{\sharp}e\|_{0}\bar{\ell}^{-k}+\bar{\ell}^{2\varepsilon-k}\|\bar{u}\|_{1+\varepsilon}^{2}\big)\leq \bar{c}_{3}(\|g-\bar{u}^{\sharp}e\|_{0}+\delta_{\ast}\lambda_{\min}(h_{\ast}))\bar{\ell}^{-k},	
\end{align*}
and
\begin{align*}
\|\bar{a}_{i}\|_{0}\leq	\bar{c}_{3}\big(\|g-\bar{u}^{\sharp}e\|_{0}+\delta_{\ast}(\lambda_{\min}(h_{\ast})+\|h_{\ast}\|_{0})\big).
\end{align*}	
Denote
\begin{align*}
\bar{u}_{0}=\bar{u}\ast\varphi_{\bar{\ell}},\quad\bar{\mu}_{0}=\bar{\ell}^{-1},\quad\bar{\mu}_{i}=K\bar{\mu}_{i-1},\;\, i=1,...,N_{\ast},
\end{align*}	
where the constant $K$ is given by \eqref{Z09} below. Then we introduce the desired Nash-spirals as follows: for $1\leq i\leq N_{\ast}$, 
\begin{align*}
	\bar{u}_{i}=\bar{u}_{i-1}+\frac{\bar{a}_{i}}{\bar{\mu}_{i}}	\big(\sin(\bar{\mu}_{i}x\cdot\nu_{i})\bar{\zeta}_{i-1}+\cos(\bar{\mu}_{i}x\cdot\nu_{i})\bar{\eta}_{i-1}\big),
\end{align*}	
where $\bar{\zeta}_{i-1},\bar{\eta}_{i-1}$ are  the normal vectors given by Lemma \ref{lem012} subject to
\begin{align*}
		\nabla \bar{u}_{i-1}^{t}\bar{\zeta}_{i-1}=\nabla \bar{u}_{i-1}^{t}\bar{\eta}_{i-1}=0,\quad \bar{\zeta}_{i-1}\cdot\bar{\eta}_{i-1}=0,\quad|\bar{\zeta}_{i-1}|=|\bar{\eta}_{i-1}|=1.
\end{align*}
By a direct computation, we obtain
\begin{align*}
\|\bar{u}_{i}\|_{k}\leq C(\bar{\mu}_{i}^{k-1}+\|\bar{u}_{i-1}\|_{k+1})\leq C\bar{\mu}_{i}^{k-1},\quad k\geq2,\;1\leq i\leq N_{\ast},
\end{align*}
and
\begin{align}\label{AQ68}
&\|\bar{u}_{N_{\ast}}\|_{1}\leq\|\bar{u}_{0}\|_{1}+\sum^{N_{\ast}}_{i=1}\big(\bar{c}_{4}\|\bar{a}_{i}\|_{0}+\bar{c}_{5}K^{-1}\big)\notag\\
&\leq\|\bar{u}\|_{1}+2N_{\ast} \bar{c}_{3}\bar{c}_{4}\big(\|g-\bar{u}^{\sharp}e\|_{0}+\delta_{\ast}(\lambda_{\min}(h_{\ast})+\|h_{\ast}\|_{0})\big)=:\|\bar{u}\|_{1}+\mathcal{G}_{\ast},
\end{align}	
provided that 
\begin{align}\label{K1}
	K\geq\frac{\bar{c}_{5}}{ \bar{c}_{3}\bar{c}_{4}\big(\|g-\bar{u}^{\sharp}e\|_{0}+\delta_{\ast}(\lambda_{\min}(h_{\ast})+\|h_{\ast}\|_{0})\big)}=:\kappa_{1}.
\end{align}
Moreover, if 
\begin{align}\label{K2}
	K\geq4\sqrt{2}N_{\ast}\bar{c}_{3}\big(\|g-\bar{u}^{\sharp}e\|_{0}+\delta_{\ast}(\lambda_{\min}(h_{\ast})+\|h_{\ast}\|_{0})\big)(\varepsilon\bar{\mu}_{0})^{-1}=:\kappa_{2},
\end{align}	
then we have
\begin{align*}
\|\bar{u}_{N_{\ast}}-\bar{u}\|_{0}\leq&\sum^{N_{\ast}}_{i=1}\|\bar{u}_{i}-\bar{u}_{i-1}\|_{0}+\|\bar{u}_{0}-\bar{u}\|_{0}\leq\sum^{N_{\ast}}_{i=1}\frac{\sqrt{2}\|a_{i}\|_{0}}{\bar{\mu}_{i}}+\|\bar{u}\|_{1}\bar{\ell}\notag\\	
\leq&\sqrt{2}N_{\ast}\bar{c}_{3}\big(\|g-\bar{u}^{\sharp}e\|_{0}+\delta_{\ast}(\lambda_{\min}(h_{\ast})+\|h_{\ast}\|_{0})\big)\bar{\mu}_{1}^{-1}+\frac{\varepsilon}{4}\leq\frac{\varepsilon}{2},
\end{align*}	
and 
\begin{align}\label{AQ69}
	\|\nabla(\bar{u}_{N_{\ast}}-\bar{u})\|_{0}\leq&\sum^{N_{\ast}}_{i=1}\|\nabla(\bar{u}_{i}-\bar{u}_{i-1})\|_{0}+\|\nabla(\bar{u}_{0}-\bar{u})\|_{0}\notag\\
	\leq&\sum^{N_{\ast}}_{i=1}8\sqrt{2}\|a_{i}\|_{0}+\|\bar{u}\|_{1+\varepsilon}\bar{\ell}^{\varepsilon}\notag\\
	\leq&16\sqrt{2}N_{\ast}\bar{c}_{3}\big(\|g-\bar{u}^{\sharp}e\|_{0}+\delta_{\ast}(\lambda_{\min}(h_{\ast})+\|h_{\ast}\|_{0})\big)=:\mathcal{K}_{\ast},	
\end{align}	
where we also used the fact that $a$ is large enough to ensure that
\begin{align*}
	\bar{\ell}\leq\min\bigg\{\frac{\varepsilon}{4\|\bar{u}\|_{1}},\bigg(\frac{8\sqrt{2}N_{0}\bar{c}_{3}\big(\|g-\bar{u}^{\sharp}e\|_{0}+\delta_{\ast}(\lambda_{\min}(h_{\ast})+\|h_{\ast}\|_{0})\big)}{\|\bar{u}\|_{1+\varepsilon}}\bigg)^{\frac{1}{\varepsilon}}\bigg\}.	
\end{align*}	
Hence, based on these above facts and in exactly the same way to Lemma \ref{lem007}, we obtain that for $i=1,...,N_{\ast}$,
\begin{align*}
\bar{u}_{i}^{\sharp}e=\bar{u}_{i-1}^{\sharp}e+\bar{a}_{i}^{2}\nu_{i}\otimes\nu_{i}+O(1)K^{-1},	
\end{align*}	
which, in combination with Lemma \ref{LEM08}, gives that
\begin{align*}
\|g-\bar{u}_{N_{\ast}}^{\sharp}e-\delta_{\ast}h_{\ast}\|_{0}\leq&\|g-g\ast\varphi_{\bar{\ell}}\|_{0}+\|g\ast\varphi_{\bar{\ell}}-\bar{u}_{N_{\ast}}^{\sharp}e-\delta_{1}h_{\ast}\|_{0}\notag\\	
\leq&\|g\|_{1}\bar{\ell}+\bar{c}_{6}K^{-1}\leq\sigma_{0}\delta_{1},
\end{align*}	
provided that 
\begin{align}\label{Z09}
K:=\frac{2\bar{c}_{6}}{\sigma_{0}\delta_{1}}\geq\max_{1\leq i\leq2}\kappa_{i},
\end{align}	
where $\kappa_{i}$, $i=1,2$ are defined by \eqref{K1}--\eqref{K2}. Remark that the inequality in \eqref{Z09} holds, since $a$ is taken sufficiently large. Observe that
\begin{align*}
\|\bar{u}_{N_{\ast}}\|_{2}\leq\bar{c}_{7}\bar{\mu}_{N_{\ast}}=\bar{c}_{7}\bar{\mu}_{0}K^{N_{\ast}}=\frac{2^{N_{\ast}+1}\bar{c}_{6}^{N_{\ast}}\bar{c}_{7}\|g\|_{1}}{(\sigma_{0}\delta_{1})^{N_{\ast}+1}}.
\end{align*}
Define
\begin{align}\label{W09}
\lambda_{\ast}=\frac{2^{N_{\ast}+1}\bar{c}_{6}^{N_{\ast}}\bar{c}_{7}\|g\|_{1}}{\delta_{\ast}^{N_{\ast}+3/2}\sigma_{0}^{N_{\ast}+1}}.	
\end{align}	
Then we deduce
\begin{align*}
	\|\bar{u}_{N_{\ast}}\|_{2}\leq\delta_{\ast}^{1/2}\lambda_{\ast}(\delta_{1}\delta_{\ast}^{-1})^{-N_{\ast}-1}=\delta_{0}^{1/2}\lambda_{0}.
\end{align*}
Then \eqref{AQ65} is proved by letting $u_{0}:=\bar{u}_{N_{\ast}}.$

\section{Further analysis for the even-dimensional case}

As seen in Corollary \ref{coro033}, when $n$ is even, if we directly apply the proof of Lemma \ref{lem007} with $\frac{n}{2}$ times, it will lead to a larger upper bound on the final constructed immersion as follows:
\begin{align*}
\|u_{m,\frac{3n}{2}}\|_{2}\leq C\delta_{m+1}^{1/2}\lambda_{m,n}\Lambda^{\frac{n}{2}}.
\end{align*}
Note that the index on $\Lambda$ reaches $\frac{n}{2}$ in the even-dimensional case, which exceeds the $\frac{n-1}{2}$ found in the odd-dimensional setting. In order to avoid a loss of regularity arising from this larger bound, a natural idea is to reduce the growth rate of the oscillation parameter used in the $\frac{n}{2}$ Nash-spirals from $\Lambda$ to $\sqrt{\Lambda}$. However, this modification brings a new difficulty: certain error terms that were previously of order $O(\delta_{m+1}\Lambda^{-1})$ are now increased to order $O(\delta_{m+1}\Lambda^{-1/2})$, becoming problematic. To tackle this obstacle, one possible approach is to proceed according to the strategy proposed by Kallen in Lemma 2.3 of \cite{K1978}, which employs the implicit function theorem to incorporate these terms as perturbations into the decomposition of symmetric positive definite matrix. For that purpose and future use, we now present a modified version of Proposition 5.4 in \cite{DI2020}, adapted to fit the iterative decomposition as follows. 

\begin{prop}\label{pro003}
Let $n\geq4$ be an even number. Assume that $1\leq\mu_{0}\leq\mu_{1}\leq\cdots\leq\mu_{n_{\ast}}$, $0<\varsigma,\delta<1$,  $0\leq n_{1}\leq n-1$, $n\leq n_{2}\leq\frac{3n}{2}-1$,
\begin{align}\label{ITE01}	
(\mu_{0}\mu_{1}^{-1})^{2}\geq\mu_{n_{2}}\mu_{n_{2}+1}^{-1},\quad\mu_{0}^{k}\geq\max\{\varsigma^{k}\mu_{n_{1}}\mu_{n_{1}+1}^{k-1},\delta^{k}\mu_{n_{2}}\mu_{n_{2}+1}^{k-1}\},\;\, k\geq 1.
\end{align}	
 Then there exists a small constant $0<\hat{\sigma}_{\ast}<\sigma_{\ast}/8$ such that for $k,l=1,...,\frac{n}{2}$, if $h,\mathfrak{T}_{k},\mathfrak{G}_{k},\Theta_{kl}\in C^{\infty}(\overline{\Omega},\mathrm{Sym}_{n})$ satisfy that $\|h\|_{i}\leq\mu_{0}^{i},\,i\geq1$,
	\begin{align*}
 \|h-h_{\ast}\|_{0}+\sum^{\frac{n}{2}}_{k=1}(\|\mathfrak{T}_{k}\|_{0}+\|\mathfrak{G}_{k}\|_{0})+\sum^{\frac{n}{2}}_{k,l=1}\|\Theta_{kl}\|_{0}+\frac{\mu_{0}}{\mu_{1}}+\sqrt{\frac{\mu_{n_{2}}}{\mu_{n_{2}+1}}}\leq2\hat{\sigma}_{\ast},
	\end{align*}
	and 
	\begin{align}\label{AQ30}
\|\mathfrak{T}_{k}\|_{i}\leq \varsigma^{i}\mu_{n_{1}}\mu_{n_{1}+1}^{i-1},\quad\sum^{\frac{n}{2}}_{k=1}\|\mathfrak{G}_{k}\|_{i}+\sum^{\frac{n}{2}}_{k,l=1}\|\Theta_{kl}\|_{i}\leq\delta^{i}\mu_{n_{2}}\mu_{n_{2}+1}^{i-1},\quad  i\geq0,	
	\end{align}				
then for any $ j\geq 0$, there exist $a^{j}=(a_{1}^{j},...,a_{n_{\ast}}^{j})\in C^{\infty}(\overline{\Omega},\mathbb{R}^{n_{\ast}})$ and $\mathcal{E}^{j}\in C^{\infty}(\overline{\Omega},\mathrm{Sym}_{n})$ such that
	\begin{align*}
		h=&\sum^{n_{\ast}}_{i=1}(a_{i}^{j})^{2}\xi_{i}\otimes\xi_{i}+\sum^{n}_{i=1}\mu_{i}^{-2}\nabla a_{i}^{j}\otimes\nabla a_{i}^{j}+\sum^{\frac{n}{2}}_{k=1}\sqrt{(a_{\bar{n}+k}^{j})^{2}-\mathfrak{T}_{k}}\,\mathfrak{G}_{k}\notag\\
		&+\sum^{\frac{n}{2}}_{k,l=1}\sqrt{(a_{\bar{n}+k}^{j})^{2}-\mathfrak{T}_{k}}\sqrt{(a_{\bar{n}+l}^{j})^{2}-\mathfrak{T}_{l}}\,\Theta_{kl}+\mathcal{E}^{j},
	\end{align*}	
with $a_{i}^{j}\geq\sqrt{\sigma_{\ast}}/2,\,1\leq i\leq n_{\ast}$ and
	\begin{align}\label{AQ31}
			\|a^{j}\|_{k}\leq C(n,j)\mu_{0}^{k},\quad\|\mathcal{E}^{j}\|_{k}\leq C(n,j)(\mu_{0}\mu_{1}^{-1})^{2(j+1)}\mu_{0}^{k},\quad k\geq0,		
	\end{align}	
where $\bar{n}:=n(n_{1}-n+1)$, $\sigma_{\ast}$ is given by Lemma \ref{lem001}.	
\end{prop}	

\begin{remark}
$\mathfrak{T}_{k}$ and $\mathfrak{G}_{k},\Theta_{kl}$ are introduced to correspond respectively to the remainder term $L_{k}(\mathcal{R})$ in \eqref{AQ908} and the perturbation terms of order $O(\delta_{m+1}\Lambda^{-1/2})$.	From the proof of Proposition \ref{pro003}, we see that the assumed condition in \eqref{ITE01} is critical to the establishment of iterative estimates in \eqref{AQ31}. In particular, the existence of small constants $\varsigma,\delta$ ensures the validity of the assumption in \eqref{ITE01}. However, precisely due to the need to introduce two such small parameters, the proof of Proposition 2.3 in \cite{CS2025} cannot be directly extended to the current setting. 
\end{remark}	

Based on the arguments in Lemma 2.3 of \cite{K1978}, Proposition 5.4 in \cite{DI2020} and Lemma 2.2 in \cite{CHI2025}, we give the proof of Proposition \ref{pro003} as follows.
\begin{proof}
{\bf Step 1.} First, using Lemma \ref{lem001}, we know that if $|h-h_{\ast}|\leq2\sigma_{\ast}$ for some small $0<\sigma_{\ast}<1$, 
	\begin{align*}
		h=\sum^{n_{\ast}}_{i=1}\bar{a}^{2}_{i}\xi_{i}\otimes\xi_{i},\quad  \bar{a}_{i}\in C^{\infty}(\overline{\Omega}),\;\,\bar{a}_{i}\geq\sqrt{\sigma_{\ast}},\quad 1\leq i\leq n_{\ast}.
	\end{align*}
For any $p>2$, denote
\begin{align*}
&S_{p}=C^{p}(\overline{\Omega},(\mathrm{Sym}_{n})^{\frac{n}{2}})\times C^{p}(\overline{\Omega},(\mathrm{Sym}_{n})^{\frac{n^{2}}{4}})\times C^{p}(\overline{\Omega},\mathrm{Sym}_{n}),\\	
&X_{p}=C^{p}(\overline{\Omega},\mathbb{R}^{n_{\ast}}),\quad Y_{p}=C^{p}(\overline{\Omega},(\mathrm{Sym}_{n})^{\frac{n}{2}})\times S_{p},\quad Z_{p}=C^{p}(\overline{\Omega},\mathrm{Sym}_{n}).
\end{align*}	
Clearly, $S_{p},X_{p},Y_{p},Z_{p}$ are all Banach spaces. Introduce the map as follows: for any $p>2$,
	\begin{align*}
		&\Phi_{p}:\,U_{p}\subset X_{p}\times Y_{p}\longrightarrow Z_{p},\\
		&(\{a_{i}\},\{\mathfrak{T}_{i}\},\{\mathfrak{G}_{i}\},\{\Theta_{ij}\},h)	\longmapsto\sum^{n_{\ast}}_{i=1}(a_{i})^{2}\xi_{i}\otimes\xi_{i}+\sum^{\frac{n}{2}}_{k=1}\sqrt{(a_{\bar{n}+k})^{2}-\mathfrak{T}_{k}}\,\mathfrak{G}_{k}-h\\
		&\;\,\quad\quad\quad\quad\quad\quad\quad\quad\quad\quad\quad\quad\quad+\sum^{\frac{n}{2}}_{k,l=1}\sqrt{(a_{2n+k})^{2}-\mathfrak{T}_{k}}\sqrt{(a_{\bar{n}+l})^{2}-\mathfrak{T}_{l}}\,\Theta_{kl},
	\end{align*}
where $U_{p}$ is an open subset of $X_{p}\times Y_{p}$ given by 
\begin{align*}
	U_{p}=\big(X_{p}\cap\{|a_{i}|>\sqrt{\sigma_{\ast}}/2:\,1\leq i\leq n_{\ast}\}\big)\times\big(Y_{p}\cap\{|\mathfrak{T}_{i}|<\sigma_{\ast}/8:\,1\leq i\leq n/2\}\big).	
\end{align*}			
A simple computation gives that $\nabla_{\{a_{i}\}}\Phi_{p}\big|_{(\{a_{i}\},\{\mathfrak{T}_{i}\},\{\mathfrak{G}_{i}\},\{\Theta_{ij}\},h)}(\{a_{i}\})$ is of class $C^{p}$ in $U_{p}$. Moreover, we have
\begin{align*}
		\Phi_{p}(\{\bar{a}_{i}\},0,0,0,h_{\ast})=0,\quad\nabla_{\{a_{i}\}}\Phi_{p}\big|_{(\{\bar{a}_{i}\},0,0,0,h_{\ast})}(\{a_{i}\})=2\sum^{n}_{i=1}\bar{a}_{i}a_{i}\xi_{i}\otimes\xi_{i}.	
\end{align*}	
Then we obtain that $\nabla_{\{a_{i}\}}\Phi_{p}\big|_{(\{\bar{a}_{i}\},0,0,0,h_{\ast})}$ is an isomorphism of $\mathbb{R}^{n_{\ast}}$ onto $\mathrm{Sym}_{n}$ due to the linear independence of $\{\xi_{i}\otimes\xi_{i}\}$. Therefore, applying the implicit function theorem, we obtain that there exists a small constant $0<\hat{\sigma}_{\ast}<\sigma_{\ast}/8$ such that if  
	\begin{align*}
		\|h-h_{\ast}\|_{0}+\sum^{\frac{n}{2}}_{k=1}(\|\mathfrak{T}_{k}\|_{0}+\|\mathfrak{G}_{k}\|_{0})+\sum^{\frac{n}{2}}_{k,l=1}\|\Theta_{kl}\|_{0}<\hat{\sigma}_{\ast},	
	\end{align*}
	then there exist functions $a_{i}^{0}(x)=\Psi_{i}(\{\mathfrak{T}_{k}\},\{\mathfrak{G}_{k}(x)\},\{\Theta_{kl}(x)\},h(x))\in C^{p}(\overline{\Omega})$, satisfying that
\begin{align}\label{AQ01}
h=&\sum^{n_{\ast}}_{i=1}(a_{i}^{0})^{2}\xi_{i}\otimes\xi_{i}+\sum^{\frac{n}{2}}_{k=1}\sqrt{(a_{\bar{n}+k}^{0})^{2}-\mathfrak{T}_{k}}\,\mathfrak{G}_{k}\notag\\
&+\sum^{\frac{n}{2}}_{k,l=1}\sqrt{(a^{0}_{\bar{n}+k})^{2}-\mathfrak{T}_{k}}\sqrt{(a^{0}_{\bar{n}+l})^{2}-\mathfrak{T}_{l}}\,\Theta_{kl},	
\end{align}	
	and
	\begin{align}
		a_{i}^{0}(x)\geq&\bar{a}_{i}(x)-|a_{i}^{0}(x)-\bar{a}_{i}(x)|\notag\\
		\geq&\bar{a}_{i}-|\Psi_{i}(\{\mathfrak{T}_{k}(x)\},\{\mathfrak{G}_{k}(x)\},\{\Theta_{kl}(x)\},h(x))-\Psi_{i}(0,0,0,h_{\ast})|\notag\\
		\geq& \sqrt{\sigma_{\ast}}-C\hat{\sigma}_{\ast}\geq	\sqrt{\sigma_{\ast}}/2,\notag\\
		\|a_{i}^{0}\|_{k}\leq& C\Big(1+\|h\|_{k}+\sum^{\frac{n}{2}}_{i=1}(\|\mathfrak{T}_{i}\|_{k}+\|\mathfrak{G}_{i}\|_{k})+\sum^{\frac{n}{2}}_{i,j=1}\|\Theta_{ij}\|_{k}\Big)\notag\\
		\leq& C\big(\mu_{0}^{k}+\max\{\varsigma^{k}\mu_{n_{1}}\mu_{n_{1}+1}^{k-1},\delta^{k}\mu_{n_{2}}\mu_{n_{2}+1}^{k-1}\}\big)\leq C\mu_{0}^{k},\quad 0\leq k\leq p,\label{A09}
	\end{align}
where we used \eqref{ITE01}--\eqref{AQ30} in the last inequality.	Denote 
	\begin{align*}
		\mathcal{E}^{0}=-\sum^{n}_{i=1}\mu_{i}^{-2}\nabla a_{i}^{0}\otimes\nabla a_{i}^{0}.
	\end{align*}	
	Then we have from \eqref{AQ01} that
	\begin{align*}
		h=&\sum^{n_{\ast}}_{i=1}(a_{i}^{0})^{2}\xi_{i}\otimes\xi_{i}+\sum^{n}_{i=1}\mu_{i}^{-2}\nabla a_{i}^{0}\otimes\nabla a_{i}^{0}+\sum^{\frac{n}{2}}_{k=1}\sqrt{(a_{\bar{n}+k}^{0})^{2}-\mathfrak{T}_{k}}\,\mathfrak{G}_{k}\notag\\
		&+\sum^{\frac{n}{2}}_{k,l=1}\sqrt{(a^{0}_{\bar{n}+k})^{2}-\mathfrak{T}_{k}}\sqrt{(a^{0}_{\bar{n}+l})^{2}-\mathfrak{T}_{l}}\,\Theta_{kl}+\mathcal{E}^{0}.
	\end{align*}	
	Using \eqref{A09}, we obtain that for $0\leq k\leq p-1,$
	\begin{align*}
		\|\mathcal{E}^{0}\|_{k}\leq&C \sum^{n}_{i=1}\mu_{i}^{-2}\|\nabla a_{i}^{0}\|_{k}	\|\nabla a_{i}^{0}\|_{0}\leq C(\mu_{0}\mu_{1}^{-1})^{2}\mu_{0}^{k}.
	\end{align*}	
	
	{\bf Step 2.} We now proceed by mathematical induction. Suppose that Proposition \ref{pro003} holds for any fixed $j\leq p-2$. Note that $a^{j}=(a_{1}^{j},...,a_{n_{\ast}}^{j})\in C^{p-j}(\overline{\Omega},\mathbb{R}^{n_{\ast}})$ and $\mathcal{E}^{j}\in C^{p-j-1}(\overline{\Omega},\mathrm{Sym}_{n})$	at this point. Then we show that it must also hold for $j+1$. For simplicity, denote
	\begin{align*}
		b^{j}_{k}=\sqrt{(a_{\bar{n}+k}^{j})^{2}-\mathfrak{T}_{k}},\quad k=1,...,\frac{n}{2}.	
	\end{align*} 
Then we have from \eqref{AQ30} and \eqref{A09} that for $0\leq i\leq p-j$,
\begin{align*}
\|b_{k}^{j}\|_{i}\leq C\big(\|a_{\bar{n}+k}^{j}\|_{i}+\|\mathfrak{T}_{k}\|_{i}\big)\leq	C(\mu_{0}^{i}+\varsigma^{i}\mu_{n_{1}}\mu_{n_{1}+1}^{i-1})\leq C\mu_{0}^{i}.
\end{align*}
Write
	\begin{align*}
		f^{j}=&\sum^{n}_{i=1}\mu_{i}^{-2}\nabla a_{i}^{j}\otimes\nabla a_{i}^{j}+\sum^{\frac{n}{2}}_{k=1}b^{j}_{k}\mathfrak{G}_{k}+\sum^{\frac{n}{2}}_{k,l=1}b^{j}_{k}b^{j}_{l}\Theta_{kl}.
	\end{align*}	
From \eqref{ITE01}--\eqref{AQ30} and \eqref{A09}, we obtain
\begin{align*}
\|f^{j}\|_{0}\leq&C\sum^{n}_{i=1}\mu_{i}^{-2}\|\nabla a_{i}^{j}\|_{0}^{2}+\sum^{\frac{n}{2}}_{k=1}\|b^{j}_{k}\|_{0}\|\mathfrak{G}_{k}\|_{0}+\sum^{\frac{n}{2}}_{k,l=1}\|b^{j}_{k}\|_{0}\|b^{j}_{l}\|_{0}\|\Theta_{kl}\|_{0}\notag\\
\leq&C(\mu_{0}\mu_{1}^{-1})^{2}+C\mu_{n_{2}}\mu_{n_{2}+1}^{-1}\leq C\hat{\sigma}_{\ast}^{2},
\end{align*}	 
and, for $0\leq k\leq p-j-1$,
	\begin{align*}
		\|f^{j}\|_{k}\leq&\sum^{n}_{i=1}\mu_{i}^{-2}\|\nabla a_{i}^{j}\|_{k}\|\nabla a_{i}^{j}\|_{0}+C\sum^{\frac{n}{2}}_{i,l=1}\big(\|b_{i}^{j}\|_{k}\|b_{l}^{j}\|_{0}+\|b_{i}^{j}\|_{0}\|b_{l}^{j}\|_{k}\big)\|\Theta_{il}\|_{0}\notag\\
		&+C\sum^{\frac{n}{2}}_{i,l=1}\|b_{i}^{j}\|_{0}\|b_{l}^{j}\|_{0}\|\Theta_{il}\|_{k}+C\sum^{\frac{n}{2}}_{i=1}\big(\|b_{i}^{j}\|_{k}\|\mathfrak{G}_{i}\|_{0}+\|b_{i}^{j}\|_{0}\|\mathfrak{G}_{i}\|_{k}\big)\notag\\
		\leq&C\mu_{1}^{-2}\mu_{0}^{k+2}+C\mu_{0}^{k}\mu_{n_{2}}\mu_{n_{2}+1}^{-1}+C\delta^{k}\mu_{n_{2}}\mu_{n_{2}+1}^{k-1}\leq C\mu_{0}^{k}.
	\end{align*}
Observe that for a sufficiently small constant $\hat{\sigma}_{\ast}>0$,
	\begin{align*}
		\|h-f^{j}-h_{\ast}\|_{0}\leq&\|h-h_{\ast}\|_{0}+\|f^{j}\|_{0}\leq C\hat{\sigma}_{\ast}\leq2\sigma_{\ast}.	
	\end{align*}	
Then from Lemma \ref{lem001}, we have
	\begin{align*}
		h-f^{j}=\sum^{n_{\ast}}_{i=1}(a_{i}^{j+1})^{2}\xi_{i}\otimes\xi_{i},\quad a_{i}^{j+1}=\sqrt{L_{i}(h-f^{j})}\geq\sqrt{\sigma_{\ast}}.	
	\end{align*}	
Denote
	\begin{align*}
		\mathcal{E}^{j+1}=&\sum^{n}_{k=1}\mu_{k}^{-2}(\nabla a_{k}^{j}\otimes\nabla a_{k}^{j}-\nabla a_{k}^{j+1}\otimes\nabla a_{k}^{j+1})\notag\\
		&+\sum^{\frac{n}{2}}_{k=1}(b_{k}^{j}-b_{k}^{j+1})\mathfrak{G}_{k}+\sum^{\frac{n}{2}}	_{k,l=1}(b_{k}^{j}b_{l}^{j}-b_{k}^{j+1}b_{l}^{j+1})\Theta_{kl}.
	\end{align*}	
A direct computation gives that for $1\leq i\leq n_{\ast}$,  $0\leq k\leq p-j-1$,
	\begin{align*}
		\|a^{j+1}_{i}\|_{k}=\|\sqrt{L_{i}(h-f^{j})}\|_{k}\leq C(\|h\|_{k}+\|f^{j}\|_{k})\leq C\mu_{0}^{k}.
	\end{align*}	
In light of the fact that 
	\begin{align*}
		a_{i}^{j+1}-a_{i}^{j}=&\sqrt{L_{i}(h-f^{j})}-\sqrt{L_{i}(h-f^{j}-\mathcal{E}^{j})}\notag\\
		=&\frac{L_{i}(\mathcal{E}^{j})}{\sqrt{L_{i}(h-f^{j})}+\sqrt{L_{i}(h-f^{j}-\mathcal{E}^{j})}},	
	\end{align*}	
	we deduce that for $0\leq k\leq p-j-1,$
	\begin{align*}
		\|a_{i}^{j}-a_{i}^{j+1}\|_{k}\leq& C\big((\|h\|_{k}+\|f^{j}\|_{k}+\|\mathcal{E}^{j}\|_{k})\|\mathcal{E}^{j}\|_{0}+\|\mathcal{E}^{j}\|_{k}\big)\leq C(\mu_{0}\mu_{1}^{-1})^{2(j+1)}\mu_{0}^{k}.
	\end{align*}	
Note that for any $c_{1}\geq c_{2}>0,\,q>0$, we have from mean value theorem that $c_{1}^{q}-c_{2}^{q}=q(c_{1}-c_{2})\xi^{q-1}$ for some $c_{2}\leq\xi\leq c_{1}$. Then we have
\begin{align*}
\|b_{i}^{j+1}-b_{i}^{j}\|_{0}=&\Big\|\sqrt{(a_{\bar{n}+i}^{j+1})^{2}-\mathfrak{T}_{i}}-\sqrt{(a_{\bar{n}+i}^{j})^{2}-\mathfrak{T}_{i}}\Big\|_{0}	\notag\\
\leq& C\big((a_{\bar{n}+i}^{j+1})^{2}-(a_{\bar{n}+i}^{j})^{2}\big)\leq C(\mu_{0}\mu_{1}^{-1})^{2(j+1)}.
\end{align*}	
Then applying mathematical induction, we obtain that for $0\leq k\leq p-j-1$,
\begin{align*}
	\|b_{i}^{j}-b_{i}^{j+1}\|_{k}\leq C(\mu_{0}\mu_{1}^{-1})^{2(j+1)}\mu_{0}^{k}.	
\end{align*}	
Combining these above facts, we deduce that for $0\leq k\leq p-j-2,$ 
	\begin{align*}
		\|\mathcal{E}^{j+1}\|_{k}\leq&C\sum^{\frac{n}{2}}_{i=1}\mu_{i}^{-2}\|\nabla a_{i}^{j}\otimes\nabla a_{i}^{j}-\nabla a_{i}^{j+1}\otimes\nabla a_{i}^{j+1}\|_{k}\notag\\ &+C\sum^{\frac{n}{2}}_{i=1}\big(\|b_{i}^{j}-b_{i}^{j+1}\|_{k}\|\mathfrak{G}_{i}\|_{0}+\|b_{i}^{j}-b_{i}^{j+1}\|_{0}\|\mathfrak{G}_{i}\|_{k}\big)\notag\\
		&+C\sum^{\frac{n}{2}}_{i,l=1}\big(\|b_{i}^{j}b_{l}^{j}-b_{i}^{j+1}b_{l}^{j+1}\|_{k}\|\Theta_{il}\|_{0}+\|b_{i}^{j}b_{l}^{j}-b_{i}^{j+1}b_{l}^{j+1}\|_{0}\|\Theta_{il}\|_{k}\big)\notag\\
		\leq& C(\mu_{0}\mu_{1}^{-1})^{2(j+2)}\mu_{0}^{k}+C(\mu_{0}\mu_{1}^{-1})^{2(j+1)}\mu_{0}^{k}\mu_{n_{2}}\mu_{n_{2}+1}^{-1}\leq C(\mu_{0}\mu_{1}^{-1})^{2(j+2)}\mu_{0}^{k},
	\end{align*}	
	where we also used \eqref{ITE01} in the last inequality and the fact that 
	\begin{align*}
		\begin{cases}
			b_{i}^{j}b_{l}^{j}-b_{i}^{j+1}b_{l}^{j+1}=b_{i}^{j}(b_{l}^{j}-b_{l}^{j+1})+b_{l}^{j+1}(b_{i}^{j}-b_{i}^{j+1}),	\\
			\nabla a_{i}^{j}\otimes\nabla a_{i}^{j}-\nabla a_{i}^{j+1}\otimes\nabla a_{i}^{j+1}=\nabla a_{i}^{j}\otimes\nabla (a_{i}^{j}-a_{i}^{j+1})+\nabla (a_{i}^{j}-a_{i}^{j+1})\otimes\nabla a_{i}^{j+1}.
		\end{cases}
	\end{align*}	
Finally, by the arbitrariness of $p>2$, we complete	the proof.

\end{proof}



\noindent{\bf{\large Acknowledgements.}} The author would like to thank Prof. C.X. Miao for his constant encouragement and useful discussions. 



\end{document}